\documentclass{article}

\usepackage{amsthm,amsmath,amssymb}

\usepackage[margin=35mm]{geometry}

\usepackage{mathtools}

\usepackage[noadjust]{cite}

\usepackage{xcolor}
\usepackage{hyperref}
\definecolor{darkgreen}{rgb}{0,0.4,0}
\definecolor{BrickRed}{rgb}{0.65,0.08,0}
\hypersetup{colorlinks=true,linkcolor=blue,citecolor=red,filecolor=BrickRed,urlcolor=darkgreen}
\newcommand{\tL}{\mathtt 1}                            
\newcommand{\tO}{\mathtt 0}                            

\DeclareMathOperator{\e}{\mathrm{e}}                   
\DeclareMathOperator{\LandauO}{\mathcal O}             
\DeclareMathOperator{\realpart}{\mathrm {Re}}
\DeclareMathOperator{\imagpart}{\mathrm{Im}}
\newcommand{\cum}{\kappa}
\def\bl{\hspace{0.5pt}\underline{\hphantom{\hspace{0.6em}}}\hspace{0.5pt}}	
\newcommand{\gammap}{\gamma^*}

\newtheorem{theorem}{Theorem}[section]
\newtheorem{proposition}[theorem]{Proposition}
\newtheorem{corollary}[theorem]{Corollary}
\newtheorem{lemma}[theorem]{Lemma}

\theoremstyle{remark}
\newtheorem*{remark}{Remark}
\newtheorem*{remarks}{Remarks}
\newtheorem*{notation}{Notation}

\numberwithin{equation}{section}

\title{The binary digits of $n+t$}

\author{
\begin{tabular}{c@{\hspace{0.5cm}}c@{\hspace{0.5cm}}c}
\begin{tabular}{c}Lukas Spiegelhofer
\\Montanuniversit\"at Leoben, Austria
\end{tabular}
&
\begin{tabular}{c}Michael Wallner\\TU Wien, Austria
\end{tabular}
\end{tabular}
}
\date{}
\begin{document}
\maketitle
\begin{abstract}
The binary sum-of-digits function $s$ counts the number of ones in the binary expansion of a nonnegative integer.
For any nonnegative integer $t$, T.~W.~Cusick defined the asymptotic density
$c_t$ of integers $n\geq 0$ such that 
\[s(n+t)\geq s(n).\]
In 2011, he conjectured that $c_t>1/2$ for all $t$ --- the binary sum of digits should, more often than not, weakly increase when a constant is added.
In this paper, we prove that there exists an explicit constant $M_0$ such that indeed $c_t>1/2$ if the binary expansion of $t$ contains at least $M_0$ maximal blocks of contiguous ones, leaving open only the ``initial cases'' --- few maximal blocks of ones --- of this conjecture.
Moreover, we sharpen a result by Emme and Hubert (2019), proving that the difference $s(n+t)-s(n)$ behaves according to a Gaussian distribution, up to an error tending to $0$ as the number of maximal blocks of ones in the binary expansion of $t$ grows.
\end{abstract}

\renewcommand{\thefootnote}{\fnsymbol{footnote}} 

\footnotetext{Lukas Spiegelhofer was supported by the Austrian Science Fund (FWF), project F5502-N26, which is a part of the Special Research Program ``Quasi Monte Carlo methods: Theory and Applications'',
and by the FWF-ANR project ArithRand (grant numbers I4945-N and ANR-20-CE91-0006).
Michael Wallner was supported by an Erwin Schr{\"o}dinger Fellowship and a stand-alone project of the Austrian Science Fund (FWF):~J~4162-N35 and~P~34142-N, respectively.}
\renewcommand{\thefootnote}{\arabic{footnote}}


\section{Introduction and main result}
The binary expansion of an integer is a fundamental concept occurring most prominently in number theory and computer science. Its close relative, the decimal expansion, is found throughout everyday life to such an extent that ``numbers'' are often understood as being the same as a string of decimal digits. However, it is difficult to argue --- mathematically --- that base ten is special; in our opinion the binary case should be considered first when a problem on digits occurs.

The basic problem we deal with is the (not yet fully understood) addition in base two.
Let us consider two simple examples: $\tL\tO + \tL = \tL\tL$ and $\tL\tL + \tL = \tL\tO\tO$. 
The difference between these two, and what makes the second one more complicated, is the occurrence of \emph{carries} and their interactions via \emph{carry propagation}.
These carries turn the problem of addition into a complicated case-by-case study and a complete characterization is unfortunately out of sight.
In order to approach this problem, we consider a parameter associated to the binary expansion --- the \emph{binary sum of digits} $s(n)$ of a nonnegative integer $n$.
This is just the number of $\tL$s in the binary expansion of $n$, and equal to the minimal number of powers of two needed to write $n$ as their sum.
While we are only dealing with this parameter instead of the whole expansion, we believe that it already contains the main difficulties caused by carry propagation.

Cusick's conjecture encodes these difficulties by simultaneously studying the sum-of-digits function of $n$ and $n+t$.
It states (private communication, 2011, 2015\footnote{The conjecture was initially termed ``Cusick problem'' or ``Question by Cusick'' in the community, but in an e-mail dated 2015 to the first author, Cusick upgraded it to ``conjecture''.}) that for all $t\geq 0$,
\begin{equation}
	\label{eqn_cusick}    
	c_t>1/2,    
\end{equation}
where
\[
c_t = \lim_{N\rightarrow \infty}
\frac 1N
\bigl \lvert\bigl\{0\leq n<N:s(n+t)\geq s(n)\bigr\}\bigr \rvert  
\] 
is the proportion of nonnegative integers $n$ such that $n+t$ contains it its binary representation at least as many $\tL$s as $n$.

This easy-to-state conjecture seems to be surprisingly hard to prove.
Moreover, it has an important connection to divisibility questions in Pascal's triangle: the formula
\begin{equation}\label{eqn_legendre}
s(n+t)-s(n)=s(t)-\nu_2\left({n+t\choose t}\right)
\end{equation}
essentially due to Legendre
links our research problem to the $2$-valuation $\nu_2$ of binomial coefficients, 
which is defined by $\nu_2(a) := \max\{e \in \mathbb{Z} : 2^e~|~a\}$.
Note also that the last term in~\eqref{eqn_legendre} is the number of carries appearing in the addition $n+t$, a result that is due to Kummer~\cite{Kummer1852}.
The strong link expressed in~\eqref{eqn_legendre}, and the combination of simplicity and complexity, has been a major motivation for our research.

In order to better understand the conjecture, we start with some simple examples.
For $t=0$ we directly get $c_0=1$. 
For $t=1$ it suffices to consider the last two digits of $n$ to obtain $c_1=3/4$. 
Note that in the two binary additions above we have $t=1$, where the first one satisfies $s(n+1)\geq s(n)$, while the second does not.
For more values of $c_t$ we used the recurrence~\eqref{eqn_density_recurrence} defined below and we verified $c_t>1/2$ for all $t\leq 2^{30}$ numerically. 
In Figure~\ref{fig:Cusick8K} we illustrate the first values of $c_t$.

\begin{figure}[ht] 
	\begin{center}%
	\includegraphics[width=\textwidth]{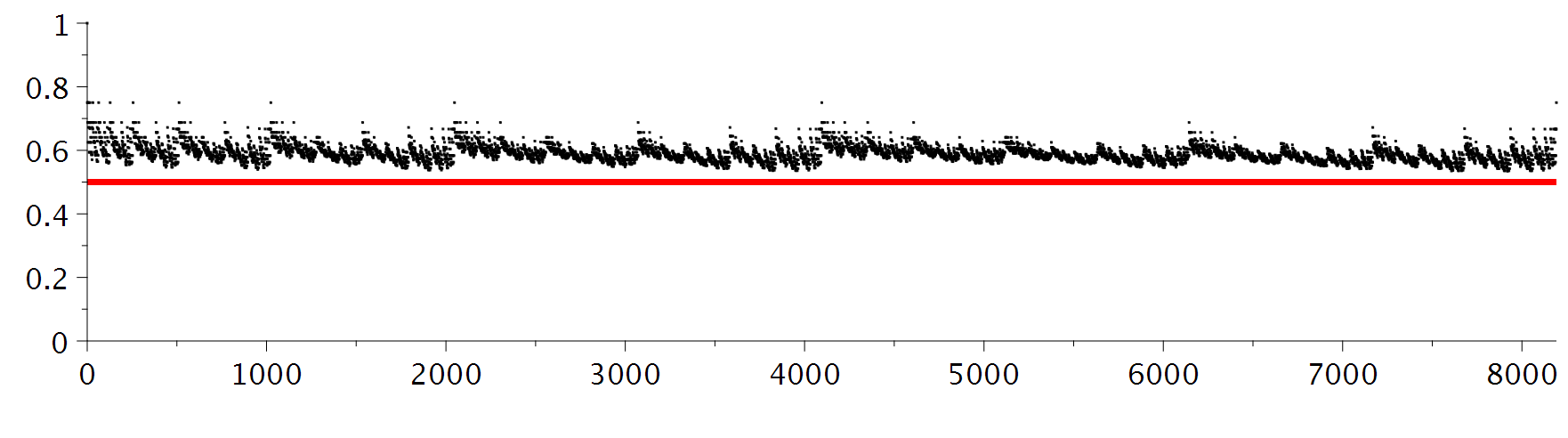}
	\caption{Cusick's conjecture states that $c_t > 1/2$ for all $t \geq 0$, which is illustrated in this figure for all $t \leq 2^{13}$, and which we computationally confirmed for all $t\leq2^{30}$.
In this paper we prove that it holds for all $t$ with sufficiently many blocks of $\tL$s, so that only finitely many classes (each class is concerned with those $t$ having a fixed number of maximal blocks of $\tL$s) remain open.
	}%
	\label{fig:Cusick8K}%
	\end{center}%
\end{figure} 

The full conjecture is still open, yet some partial results have been obtained~\cite{DKS2016,EH2018,EH2018b,EP2017,S2019,S2020}.
Among these, we want to stress a central limit-type result by Emme and Hubert~\cite{EH2018},
a lower bound due to the first author~\cite{S2020}, 
and
an almost-all result by Drmota, Kauers, and the first author~\cite{DKS2016} stating that 
for all $\varepsilon>0$, we have
\[\lvert\{t<T:1/2<c_t<1/2+\varepsilon\}\rvert=T-\LandauO\left(\frac{T}{\log T} \right).\]
(The symbol~$\LandauO$ is used for Big O notation throughout this paper.)
Moreover, Cusick's conjecture is strongly connected to the \emph{Tu--Deng conjecture}~\cite{TD2011,TD2012} in cryptography, which is also still open, yet with some partial results~\cite{CLS2011,DY2012,F2012,FRCM10,SW2019,TD2011}.
We presented this connection in~\cite{SW2019}, in which we
proved an almost-all result for the Tu--Deng conjecture 
and where we showed that the full Tu--Deng conjecture implies Cusick's conjecture.

The main theorem of this paper is the following near-solution to Cusick's conjecture, which significantly improves the previous results.
Note that it happens repeatedly that difficult conjectures are (more easily) provable for sufficiently large integers and recently even two more important ones have been resolved in this manner: Sendov's conjecture~\cite{T2021} and the Erd\H{o}s--Faber--Lov\'asz conjecture~\cite{KKKMO2021}.
Our method will combine several techniques such as recurrence relations, cumulant generating functions, and integral representations.

\begin{theorem}\label{thm_main}
There exists a constant $M_0$ with the following property:
If the natural number~$t$ has at least $M_0$ maximal blocks of $\tL$s in its binary expansion, then $c_t>1/2$.
\end{theorem}

\begin{remark}
We note the important observation that all constants in this paper could be given numerical values by following our proofs.
In order to keep the technicalities at a minimum, we decided not to compute them explicitly.
In this paper, we do not rely on arguments making it impossible to extract explicit values for our constants (such as certain proofs by contradiction).
We are dealing with \emph{effective} results, without giving a precise definition of this term.
\end{remark}

The central objects to tackle the conjecture are the asymptotic densities
\[\delta(j,t)=\lim_{N\rightarrow\infty}\frac 1N\#\,\bigl\{0\leq n<N
:s(n+t)-s(n)=j\bigr\},\]
where $j\in\mathbb Z$.
The limit exists in our case; see B\'esineau~\cite{B1972}.
These densities lead to the useful decomposition
\begin{equation}\label{eqn_ct_sum}
	c_t=\sum_{j\geq 0}\delta(j,t).
\end{equation}
The sum on the right hand side is in fact finite, since $\delta(j,t)=0$ for $j>s(t)$, which follows from~\eqref{eqn_legendre}.
Therefore we get equality in~\eqref{eqn_ct_sum} ---
asymptotic densities are finitely additive.

Distinguishing between even and odd cases, one can show that the values $\delta(k,t)$ satisfy the following recurrence~\cite{DKS2016,S2019,S2020}:
\begin{equation}\label{eqn_deltaj1}
\delta(j,1)=\begin{cases}0,&j>1;\\2^{j-2},&j\leq 1,\end{cases}
\end{equation}
and for $t\geq 0$,
\begin{equation}\label{eqn_density_recurrence}
\begin{aligned}
\delta(j,2t)&=\delta(j,t),\\
\delta(j,2t+1)&=\frac 12 \delta(j-1,t)+\frac 12\delta(j+1,t+1).
\end{aligned}
\end{equation}
In particular, the recurrence shows that $\delta(\bl,t)$ is a probability mass function for each $t$:
\begin{equation}\label{eqn_delta_summable}
\sum_{j\in\mathbb Z}\delta(j,t)=1,
\end{equation}
and $\delta(j,t)\geq 0$ by definition. 
Furthermore, the set
\[
\{n\in\mathbb N:s(n+t)-s(n)=j\}
\]
defining $\delta(j,t)$ is a finite union of arithmetic progressions $a+2^m\mathbb N$,
which can be seen along the same lines.

Our second main result gives an asymptotic formula for the densities $\delta(j,t)$ and is obtained in the course of establishing Theorem~\ref{thm_main}.
\begin{theorem}\label{thm_normal}
For integers $t\geq 1$, let us define
\begin{align*}
\cum_2(1)=2;\qquad \cum_2(2t)=\cum_2(t);\qquad
\cum_2(2t+1)=\frac{\cum_2(t)+\cum_2(t+1)}2+1.
\end{align*}
If the positive integer $t$ has $M$ maximal blocks of $\tL$s in its binary expansion, and $M$ is larger than some constant $M_0$, then we have
\begin{equation*}
\delta(j,t)
=\frac 1{\sqrt{2\pi\cum_2(t)}}\exp\left(-\frac{j^2}{2\cum_2(t)}\right)
+\LandauO\bigl(M^{-1}(\log M)^4\bigr)
\end{equation*}
for all integers $j$.
The multiplicative constant in the error term can be made explicit.
\end{theorem}
Concerning the effectiveness of the constants, we refer to the remark after Theorem~\ref{thm_main}.
We will see in Corollary~\ref{cor_aj_bounds} and in Lemma~\ref{lem_a2_lower_bound} that
\[M\leq \cum_2(t) \leq CM\]
for some constant $C$.
Therefore, the main term dominates the error term for large $M$ if
\[\lvert j\rvert\leq \frac 12 \sqrt{M\log M}.\]
Note that the factor $1/2$ is arbitrary and any value $\rho<1$ is good enough
(for $M$ larger than some bound depending on $\rho$).
Moreover, in the statement of Theorem~\ref{thm_normal}, the lower bound $M > M_0$ is, in fact, not needed, as it can be taken care of by the constant $C$ in the error term.
Simply choose $C$ so large that the error term is greater than $1$ (for example, $C=M_0$ is sufficient, since $\delta(j,t) \leq 1$).
We decided to keep the theorem as it is,
since we feel that increasing a constant only for reasons of brevity is somewhat artificial.

Without giving a full proof we note that, by summation, this theorem can be used for proving a statement comparing $\Delta(j,t)=\sum_{j'\geq j}\delta(j',t)$ and the Gaussian $\Phi\bigl(-j/\sqrt{\cum_2(t)}\bigr)$.
This leads to a sharpening of the main result in Emme and Hubert~\cite{EH2018}.
By summing the asymptotic formula in Theorem~\ref{thm_normal} from $0$ to $\sqrt{M}\log M$, we also obtain the following corollary.
\begin{corollary}
There exists a constant $C$ such that
\[c_t\geq 1/2-CM^{-1/2}\bigl(\log M\bigr)^5\]
for all $t\geq 1$, where $M$ is the number of maximal blocks of $\tL$s in $t$.
\end{corollary}
The proof is straightforward, and left to the reader.
This corollary is weaker than Theorem~\ref{thm_main}, but we stated it here since it gives a quantitative version of the main theorem in~\cite{S2020}.

\begin{notation}
In this paper, $0\in\mathbb N$.
We will use Big O notation, employing the symbol~$\LandauO$.
We let $\e(x)$ denote $e^{2\pi ix}$ for real $x$.
In our calculations, the number $\pi$ will often appear with a factor $2$.
Therefore we use the abbreviation $\tau=2\pi$.

We consider blocks of $\tO$s or $\tL$s in the binary expansion of an integer $t\in\mathbb N$.
Writing ``block of $\tL$s of length $\nu$ in $t$'', we always mean a maximal subsequence $\varepsilon_\mu=\varepsilon_{\mu+1}=\cdots=\varepsilon_{\mu+\nu-1}=1$ (where maximal means that $\varepsilon_{\mu+\nu}=0$ and either $\mu=0$ or $\varepsilon_{\mu-1}=0$).
  ``Blocks of $\tO$s of length $\nu$ in $t$'' are subsequences $\varepsilon_\mu=\cdots=\varepsilon_{\mu+\nu-1}=0$ such that $\varepsilon_{\mu+\nu}=1$ and either $\mu=0$ or $\varepsilon_{\mu-1}=1$.
We call blocks of zeros bordered by $\tL$s on both sides ``inner blocks of $\tO$s''.
For example, $2^kn$ and $n$ have the same number of inner blocks of $\tO$s.
The \emph{number of blocks in} $t$ is the sum of the number of blocks of $\tL$s and the number of blocks of $\tO$s.

All constants in this paper are absolute and effective. The letter $C$ is often used for constants; occurrences of $C$ at different positions need not necessarily designate the same value.
\end{notation}

In the remainder we give the proof of our main result, Theorem~\ref{thm_main}, followed by the proof of Theorem~\ref{thm_normal}.

\subsubsection*{Acknowledgments.}
On one of his first days as a PhD student in 2011, the first author was introduced to Cusick's conjecture by Johannes F.\ Morgenbesser, whom he wishes to thank at this point.
This conjecture has ever since been a source of inspiration and motivation to him.
We also wish to thank Michael Drmota, Jordan Emme, Wolfgang Steiner, and Thomas Stoll for fruitful discussions on the topic.
Finally, we thank Thomas W.~Cusick for constant encouragement and interest in our work.
\section{Proof of the main theorem}

The proof of our main Theorem~\ref{thm_main} is split into several parts. 
The main idea is to work with the cumulant generating function of the probability distribution given by the densities $\delta(j,t)$, which we define in Section~\ref{sec:charcumgf}. 
The crucial observation later on is that it is sufficient to work with an approximation using only the cumulants up to order $5$.
This approximation is analyzed in Section~\ref{sec:approx} and used in Section~\ref{sec:intct} inside an explicit integral representation of $c_t$ to prove our main result up to an exceptional set of $t$s.
It remains to prove that these exceptional values, which are defined by the cumulants of order $2$ and $3$, satisfy an inequality involving the cumulants of order $4$ and $5$. 
For this reason, we needed to choose an approximation of the cumulant generating function up to order $5$.
Thus, in Section~\ref{sec:determinintexceptionalset} we determine this exceptional set and in Section~\ref{sec:boundsK4K5} we prove bounds on the cumulants of order $4$ and $5$. 
Finally, in Section~\ref{sec:endmainproof} we combine all ingredients to prove the inequality.

\subsection{Characteristic function 
and cumulant generating function}
\label{sec:charcumgf}

We begin with the definition of the characteristic function of the probability distribution given by the densities $\delta(j,t)$.
In particular, we use the following variant, involving a scaling factor $\tau=2\pi$.
For $t\geq 0$ and $\vartheta\in\mathbb R$ we define
\begin{equation*}
\gamma_t(\vartheta)=\sum_{j\in\mathbb Z}\delta(j,t)\e(j\vartheta).
\end{equation*}
Since $\delta(\bl,t)$ defines a probability distribution and $\lvert\e(x)\rvert\le1$ for real $x$,
we may interchange summation and integration by the dominated convergence theorem:
\begin{equation}\label{eqn_gamma_integral}
\begin{aligned}
\delta(j,t)&=\sum_{k\in\mathbb Z}\delta(k,t)\cdot
\left\{\begin{array}{ll}1,&k=j;\\0,&k\neq j\end{array}\right\}
=
\sum_{k\in\mathbb Z}\delta(k,t)
\int_{-1/2}^{1/2}\e((k-j)\vartheta)\,\mathrm d\vartheta
\\&=
\int_{-1/2}^{1/2}
\e(-j\vartheta)\sum_{k\in\mathbb Z}\delta(k,t)\e(k\vartheta)\,\mathrm d\vartheta
=\int_{-1/2}^{1/2} \gamma_t(\vartheta)\e(-j\vartheta)\,\mathrm d\vartheta.
\end{aligned}
\end{equation}
The recurrence~\eqref{eqn_density_recurrence} directly carries over to the characteristic functions.
For all $t\geq 0$, we have
\begin{equation}\label{eqn_gamma_recurrence}
\begin{aligned}
\gamma_{2t}(\vartheta)&=\gamma_t(\vartheta),\\
\gamma_{2t+1}(\vartheta)&=\frac{\e(\vartheta)}2\gamma_t(\vartheta)+\frac{\e(-\vartheta)}2\gamma_{t+1}(\vartheta),
\end{aligned}
\end{equation}
and in particular
\begin{equation}\label{eqn_gamma_1}
\gamma_1(\vartheta)=\frac{\e(\vartheta)}{2-\e(-\vartheta)}.
\end{equation}
Therefore, for all $t\geq 1$, we have
\[\gamma_t(\vartheta)=\omega_t(\vartheta)\gamma_1(\vartheta),\]
where $\omega_t$ is a trigonometric polynomial such that $\omega_t(0)=1$.
These polynomials satisfy the same recurrence relation as $\gamma_t$.
In particular, noting also that the denominator $2-\e(-\vartheta)$ is nonzero near $\vartheta=0$, we have $\realpart \gamma_t(\vartheta)>0$ for $\vartheta$ in a certain disk
\[D_t=\{\vartheta\in\mathbb C:\lvert \vartheta\rvert<r(t)\},\]
where $r(t)>0$. It follows that
\begin{equation}\label{eqn_dt_def}
K_t = \log\circ\,\gamma_t
\end{equation}
is analytic in $D_t$ and therefore there exist complex numbers $\cum_j(t)$ for $j\in\mathbb N$ such that
\begin{equation}
\label{eq_gammat_cum}
\gamma_t(\vartheta)
=\exp(K_t(\vartheta))
=\exp\left(\sum_{j\geq 0}\frac{\cum_j(t)}{j!}(i\tau\vartheta)^j\right)
\end{equation}
for all $\vartheta\in D_t$.
These numbers $\cum_j(t)$ are the \emph{cumulants} of the probability distribution defined by $\delta(\bl,t)$ (up to a scaling by $\tau$); see, e.g.,~\cite{B2012}.
They are real numbers since characteristic functions are Hermitian: $\gamma_t(\vartheta)=\overline{\gamma_t(-\vartheta)}$.
The real-valuedness also follows directly from the fact that cumulants are defined via the logarithm of the moment generating function, which has real coefficients.
The cumulant $\cum_2(t)$ is the variance: we have
\begin{equation}\label{eqn_A2_moment}
\cum_2(t)=\sum_{j\in\mathbb Z}j^2\delta(j,t).
\end{equation}
For $t=0$, we have $\cum_j(t)=0$ for all $j\geq 0$, as $\delta(k,0)=1$ if $k=0$ and $\delta(k,0)=0$ otherwise.
The recurrence~\eqref{eqn_gamma_recurrence} shows that
\[\gamma_t(\vartheta)=1+\mathcal O(\vartheta^2)\]
at $0$, which implies $\cum_0(t)=\cum_1(t)=0$.
Let us write
\begin{equation}\label{eqn_abcA_rewriting}
x_j=\cum_j(t),\quad y_j=\cum_j(t+1),\quad\mbox{and}\quad
z_j=\cum_j(2t+1).
\end{equation}
Next, we will express the coefficients $z_j$ as functions of the coefficients $x_j$ and $y_j$.
Therefore we substitute the cumulant representation from~\eqref{eq_gammat_cum} for $\gamma_t(\vartheta)$ into the recurrence~\eqref{eqn_gamma_recurrence} and obtain that these quantities are related via the fundamental identity
{ \everymath={\displaystyle}
\begin{equation}\label{eqn_coeff_rec_exp}
\begin{aligned}
\hspace{4em}&\hspace{-4em}
\exp\left(\frac{z_2}2(i\tau\vartheta)^2+\frac{z_3}6(i\tau\vartheta)^3+\cdots\right)
\\&
\begin{array}{ll@{\hspace{1mm}}l@{\hspace{0em}}l@{\hspace{0em}}l@{\hspace{0em}}l}
&=&\frac12
\exp\Bigl(&i\,\!\tau\vartheta+\frac{x_2}2&(i\tau\vartheta)^2+\frac{x_3}6&(i\tau\vartheta)^3+\cdots\Bigr)
\\[2mm]&
+&\frac 12
\exp\Bigl(-&i\,\!\tau\vartheta+\frac{y_2}2&(i\tau\vartheta)^2+\frac{y_3}6&(i\tau\vartheta)^3+\cdots\Bigr),
\end{array}
\end{aligned}
\end{equation}%
}%
valid for $\vartheta\in D=D_t\cap D_{t+1}\cap D_{2t+1}$.
From this equation, we derive the following lemma by comparing coefficients of the appearing analytic functions.
\begin{lemma}\label{lem_exponent_rec}
Assume that $t\geq 0$ and let $x_j$, $y_j$, and $z_j$ be defined by~\eqref{eqn_abcA_rewriting}.
We have
\begin{align}
z_2&=\frac{x_2+y_2}2+1;\label{eqn_coeff_2}\\
z_3&=\frac{x_3+y_3}2+\frac32(x_2-y_2);\label{eqn_coeff_3}\\
z_4&=\frac{x_4+y_4}2+2(x_3-y_3)+\frac 34(x_2-y_2)^2-2;\label{eqn_coeff_4}\\
z_5&=\frac{x_5+y_5}2+\frac 52(x_4-y_4)+\frac 52(x_2-y_2)(x_3-y_3)-10(x_2-y_2).
\label{eqn_coeff_5}
\end{align}
In particular,
\begin{equation}\label{eqn_fubini}
\cum_2(1)=2,\quad \cum_3(1)=-6,\quad \cum_4(1)=26,\quad \cum_5(1)=-150.
\end{equation}
\end{lemma}
\begin{proof}
Extracting the coefficient of $\vartheta^2$ in~\eqref{eqn_coeff_rec_exp}, we obtain
\begin{align*}z_2&=\frac{1}{(i\tau)^2}[\vartheta^2]
\left(
1+i\,\tau\vartheta+\frac{x_2}2(i\tau \vartheta)^2+\frac 12\bigl(i\,\tau\vartheta+\frac{x_2}2(i\tau\vartheta)^2\bigr)^2\right.
\\&+\left.
1-i\,\tau\vartheta+\frac{y_2}2(i\tau \vartheta)^2+\frac 12\bigl(-i\,\tau\vartheta+\frac{y_2}2(i\tau\vartheta)^2\bigr)^2
\right)
=\frac{x_2+y_2}2+1,
\end{align*}
where $[x^k] \sum f_k x^k = f_k$ denotes the coefficient extraction operator and
this gives~\eqref{eqn_coeff_2}.

Similarly, we handle the higher coefficients.
We proceed with $[\vartheta^3]K_t(\vartheta)$.
From~\eqref{eqn_coeff_rec_exp} we obtain by collecting the cubic terms
\begin{align*}
z_3&=\frac{3}{(i\tau)^3}
\left(
\frac{x_3}6(i\tau)^3+2\frac 12\frac{x_2}2(i\tau)^3+\frac 16(i\tau)^3
+\frac{y_3}6(i\tau)^3-2\frac12\frac{y_2}2(i\tau)^3-\frac 16(i\tau)^3
\right)
\\&=\frac{x_3+y_3}2+\frac 32(x_2-y_2),
\end{align*}
which is~\eqref{eqn_coeff_3}.
For the next coefficient $[\vartheta^4]K_t(\vartheta)$, we have to take the quadratic term of the exponential on the left hand side of~\eqref{eqn_coeff_rec_exp} into account.
This yields, inserting the recurrence for $z_2$ obtained before,
\begin{multline*}
\bigl[\vartheta^4\bigr]
\exp\left(\frac{z_2}2(i\tau\vartheta)^2+\frac{z_3}6(i\tau\vartheta)^3+\frac{z_4}{24}(i\tau\vartheta)^4\right)
=
\tau^4\left(\frac{z_4}{24}+\frac{z_2^2}8\right)
\\=
\tau^4\left(\frac{z_4}{24}+\frac 18+\frac{x_2+y_2}8+\frac {(x_2+y_2)^2}{32}\right).
\end{multline*}
The coefficient of $\vartheta^4$ of the right hand side of~\eqref{eqn_coeff_rec_exp} gives, collecting the quartic terms,
\begin{align*}
\hspace{4em}&\hspace{-4em}
\frac {\tau^4}2\left(
\frac{x_4}{24}
+\frac 12\left(\frac{x_3}3+\frac{x_2^2}4\right)
+\frac 16\biggl(3\frac{x_2}2\biggr)
+\frac 1{24}
\right.
\\&
\left.
+\frac{y_4}{24}
+\frac 12\left(-\frac{y_3}3+\frac{y_2^2}4\right)
+\frac 16\left(3\frac{y_2}2\right)
+\frac 1{24}
\right)
\\&=\tau^4
\left(
\frac{x_4+y_4}{48}
+\frac{x_3-y_3}{12}
+\frac{x_2^2+y_2^2}{16}
+\frac{x_2+y_2}{16}
+\frac 1{24}
\right).
\end{align*}
Equation~\eqref{eqn_coeff_4} follows.
Finally, we need the quintic terms.
The left hand side of~\eqref{eqn_coeff_rec_exp} yields
\begin{align*}
\hspace{4em}&\hspace{-4em}
\bigl[\vartheta^5\bigr]
\exp\left(\frac{z_2}2(i\tau\vartheta)^2+\frac{z_3}6(i\tau\vartheta)^3+\frac{z_4}{24}(i\tau\vartheta)^4+\frac{z_5}{120}(i\tau\vartheta)^5\right)
=
(i\tau)^5\left(\frac{z_5}{120}+\frac{z_2z_3}{12}\right)
\\&=
(i\tau)^5\left(\frac{z_5}{120}+\frac 1{12}\left(\frac{x_2+y_2}2+1\right)\left(\frac{x_3+y_3}2+\frac32(x_2-y_2)\right)\right),
\end{align*}
while the right hand side of~\eqref{eqn_coeff_rec_exp} yields
\begin{align*}
\hspace{1.5em}&\hspace{-1.5em}
\frac{(i\tau)^5}2
\left(
\frac{x_5}{120}+\frac 12\left(2\frac{x_2x_3}{12}+2\frac{x_4}{24}\right)
+\frac 16\left(3\frac{x_3}6+3\frac{x_2^2}4\right)
+\frac 1{24}\left(4\frac{x_2}2\right)
+\frac 1{120}
\right.\\&\left.
+\frac{y_5}{120}+\frac 12\left(2\frac{y_2y_3}{12}-2\frac{y_4}{24}\right)
+\frac 16\left(3\frac{y_3}6-3\frac{y_2^2}4\right)
-\frac 1{24}\left(4\frac{y_2}2\right)
-\frac 1{120}
\right)
\\&=
(i\tau)^5\left(
\frac{x_5+y_5}{240}
+\frac {x_4-y_4}{48}
+\frac{x_3+y_3}{24}
+\frac{x_2x_3+y_2y_3}{24}
+\frac{x_2^2-y_2^2}{16}
+\frac{x_2-y_2}{24}
\right),
\end{align*}
which implies~\eqref{eqn_coeff_5} after a short calculation.
Finally, we compute the values $\cum_2(1), \dots, \cum_5(1)$ by substituting $t=0$ in~\eqref{eqn_coeff_2}--\eqref{eqn_coeff_5}.
\end{proof}

In the following, we are not concerned with the original definition of $\cum_j$, involving a disk $D_t$ with potentially small radius.
Instead, we only work with the recurrences~\eqref{eqn_coeff_2}--\eqref{eqn_coeff_5}, which we restate here explicitly as a main result of this section:
\begin{align}
\cum_j(2t)&=\cum_j(t)\quad\mbox{for all }j\geq 0;\nonumber\\
\cum_2(2t+1)&=\frac12 \bigl(\cum_2(t)+\cum_2(t+1)\bigr)+1;\nonumber\\
\cum_3(2t+1)&=\frac12 \bigl(\cum_3(t)+\cum_3(t+1)\bigr)+\frac 32\bigl(\cum_2(t)-\cum_2(t+1)\bigr);\nonumber\\
\cum_4(2t+1)&=\frac12 \bigl(\cum_4(t)+\cum_4(t+1)\bigr)+2\bigl(\cum_3(t)-\cum_3(t+1)\bigr)\label{eqn_dt_rec}
\\&+\frac34\bigl(\cum_2(t)-\cum_2(t+1)\bigr)^2-2;\nonumber\\
\cum_5(2t+1)&=\frac12 \bigl(\cum_5(t)+\cum_5(t+1)\bigr)+\frac 52\bigl(\cum_4(t)-\cum_4(t+1)\bigr)\nonumber
\\&+\frac 52\bigl(\cum_2(t)-\cum_2(t+1)\bigr)\bigl(\cum_3(t)-\cum_3(t+1)\bigr)\nonumber
\\&-10\bigl(\cum_2(t)-\cum_2(t+1)\bigr),\nonumber
\end{align}
for all integers $t\geq 0$.
Note that $\cum_2(t)$ is obviously nonnegative, since it is a variance; this can also easily be seen from this recurrence.
\begin{remarks}
Let us discuss some properties and other appearances of $\cum_j(t)$.
\begin{enumerate}
\item
The sequence $\cum_2$ is $2$-\emph{regular}~\cite{AlloucheShallit1992, AS2003, AlloucheShallit2003}.
More precisely, we define
\begin{equation}\label{eqn_transition_matrices}
B_0=\left(\begin{matrix}
1&0&0\\
1/2&1/2&1\\
0&0&1
\end{matrix}\right),
\qquad
B_1=\left(\begin{matrix}
1/2&1/2&1\\
0&1&0\\
0&0&1
\end{matrix}\right)
\end{equation}
and
\[S(n)=\left(\begin{matrix}S_1(n)\\S_2(n)\\S_3(n)\end{matrix}\right)
=\left(\begin{matrix}\cum_2(n)\\\cum_2(n+1)\\1\end{matrix}\right).
\]
Then for all $n\geq 0$, the recurrence yields
\begin{equation}\label{eqn_k2_regular}
S(2n)=B_0S(n), \qquad S(2n+1)=B_1S(n).
\end{equation}
Thus $\kappa_2$ is $2$-regular, compare to~\cite[Theorem~2.2, item (e)]{AlloucheShallit1992}.

In this manner, we can also prove $2$-regularity of $\kappa_3,\kappa_4,\kappa_5$.
Considering for example the case $\kappa_5$, we introduce a sequence $S_\ell$ for each term that occurs in one of the recurrence formulas~\eqref{eqn_dt_rec}, such as $\kappa_2(n)\kappa_3(n+1)$; we see that it is sufficient to consider two $16\times 16$-matrices.
\item
The sequence $d_t=\cum_2(t)/2$
appears in another context too: it is the \emph{discrepancy of the van der Corput sequence}~\cite{DLP2005,S2018}, and it satisfies $d_1=1$, $d_{2t}=d_t$, $d_{2t+1}=(d_t+d_{t+1}+1)/2$.
We do not know yet if this connection between our problem and discrepancy is a meaningful one. 
After all, it is no big surprise that one of the simplest $2$-regular sequences
occurs in two different problems concerning the binary expansion.

\item By the same method of proof (or alternatively, by concatenating the power series for $\log$ and $\gamma_t(\vartheta)$) the list in Lemma~\ref{lem_exponent_rec} can clearly be prolonged indefinitely.
For the proof of our main theorem, however, we only need the terms up to $\cum_5$.
Without giving a rigorous proof, we note that this also shows that $\kappa_j$ is $2$-regular for all $j\geq 0$.
Note the important property that lower cumulants always appear as differences; we believe that this behavior persists for higher cumulants.
\item More explicit values of $\cum_j(1)$ can be easily computed from the closed form~\eqref{eqn_gamma_1}. 
Note that by~\eqref{eqn_deltaj1} we know that these numbers are the cumulants of a geometric distribution with parameter $p=1/2$ and given by the OEIS sequence~\href{http://oeis.org/A000629}{\texttt{A000629}} with many other combinatorial connections.
\end{enumerate}
\end{remarks}

In the next section we analyze an approximation of the cumulant generating function $\gamma_t(\vartheta)$ anticipating the fact that it captures all important properties for the subsequent proof.

\subsection{An approximation of the cumulant generating function}
\label{sec:approx}
Let us define the following approximation of $\gamma_t$.
Set
\begin{equation}\label{eqn_gammaprime_def}
\gammap_t(\vartheta)=\exp\left(\sum_{2\leq j\leq 5}\frac{\cum_j(t)}{j!}(i\tau\vartheta)^j\right).
\end{equation}
We are going to replace $\gamma_t$ by $\gammap_t$, and for this purpose we have to bound the difference
\[\widetilde\gamma_t(\vartheta)=\gamma_t(\vartheta)-\gammap_t(\vartheta).\]
Clearly, we have $\widetilde\gamma_{2t}(\vartheta)=\widetilde\gamma_t(\vartheta)$.
Moreover,
\begin{equation}\label{eqn_tilde_gamma_rec}
\begin{aligned}
\widetilde\gamma_{2t+1}(\vartheta)&=\frac{\e(\vartheta)}2\bigl(\widetilde\gamma_t(\vartheta)+\gammap_t(\vartheta)\bigr)
+\frac{\e(-\vartheta)}2\bigl(\widetilde\gamma_{t+1}+\gammap_{t+1}(\vartheta)\bigr)-\gammap_{2t+1}(\vartheta)\\
&=\frac{\e(\vartheta)}2\widetilde\gamma_t(\vartheta)+\frac{\e(-\vartheta)}2\widetilde\gamma_{t+1}(\vartheta)+\xi_t(\vartheta),
\end{aligned}
\end{equation}
where
\begin{equation}\label{eqn_xi_def}
\xi_t(\vartheta)=
\frac{\e(\vartheta)}2
\gammap_t(\vartheta)
+\frac{\e(-\vartheta)}2
\gammap_{t+1}(\vartheta)
-\gammap_{2t+1}(\vartheta).\end{equation}
We prove the following rough bounds on differences of the cumulants $\cum_j$.
\begin{lemma}\label{lem_aj_diff_bounds}
We have
\begin{align}
\lvert \cum_2(t+1)-\cum_2(t)\rvert&\leq 2;\label{eqn_a2_difference_bound}\\
\lvert \cum_3(t+1)-\cum_3(t)\rvert&\leq 6;\label{eqn_a3_difference_bound}\\
\lvert \cum_4(t+1)-\cum_4(t)\rvert&\leq 28;\label{eqn_a4_difference_bound}\\
\lvert \cum_5(t+1)-\cum_5(t)\rvert&\leq 240.\label{eqn_a5_difference_bound}
\end{align}
\end{lemma}
\begin{proof}
We prove these statements by induction, inserting the recurrences~\eqref{eqn_dt_rec}.
We have
\[\cum_2(2t+1)-\cum_2(2t)=
\frac{\cum_2(t)+\cum_2(t+1)}2+1-\cum_2(t)
=
\frac{\cum_2(t+1)-\cum_2(t)}2+1
\]
and
\[\cum_2(2t+2)-\cum_2(2t+1)=
\cum_2(t+1)-\frac{\cum_2(t)+\cum_2(t+1)}2+1
=
\frac{\cum_2(t+1)-\cum_2(t)}2-1.
\]
Then, by induction, the first statement is an easy consequence.
Next, we consider the second inequality. 
From~\eqref{eqn_dt_rec} we get
\begin{align*}
\cum_3(2t+1)-\cum_3(2t)
&=\frac{\cum_3(t+1)-\cum_3(t)}2-\frac 32\bigl(\cum_2(t+1)-\cum_2(t)\bigr),\\
\cum_3(2t+2)-\cum_3(2t+1)
&=\frac{\cum_3(t+1)-\cum_3(t)}2+\frac 32\bigl(\cum_2(t+1)-\cum_2(t)\bigr),
\end{align*}
and using the first part and induction, the claim follows.
Concerning~\eqref{eqn_a4_difference_bound},
\begin{equation*}
\begin{aligned}
\cum_4(2t+1)-\cum_4(2t)
&=\frac{\cum_4(t+1)-\cum_4(t)}2+2\bigl(\cum_3(t)-\cum_3(t+1)\bigr)
\\&+\frac34\bigl(\cum_2(t)-\cum_2(t+1)\bigr)^2-2,
\end{aligned}
\end{equation*}
and the last three summands add up to a value bounded by $14$ in absolute value, using the first two estimates and the fact that all cumulants are real numbers.
An analogous statement for $\cum_4(2t+2)-\cum_4(2t+1)$ holds.
This implies the third line.
Finally,
\begin{equation*}
\begin{aligned}
\hspace{3em}&\hspace{-3em}
\cum_5(2t+1)-\cum_5(2t)
=\frac{\cum_5(t+1)-\cum_5(t)}2+\frac52\bigl(\cum_4(t)-\cum_4(t+1)\bigr)
\\&+\frac52\bigl(\cum_2(t)-\cum_2(t+1)\bigr)\bigl(\cum_3(t)-\cum_3(t+1)\bigr)-10\bigl(\cum_2(t)-\cum_2(t+1)\bigr),
\end{aligned}
\end{equation*}
and the sum of the last three summands is bounded by $120$ in absolute value.
In complete analogy to the above, this implies~\eqref{eqn_a5_difference_bound}.
\end{proof}
\begin{corollary}\label{cor_aj_bounds}
There exists a constant $C$ such that
for all $t$ having $M$ blocks of $\tL$s we have
\[\lvert \cum_2(t)\rvert\leq CM,\quad
\lvert \cum_3(t)\rvert\leq CM,\quad
\lvert \cum_4(t)\rvert\leq CM,\quad
\lvert \cum_5(t)\rvert\leq CM.\]

\end{corollary}
\begin{proof}
We proceed by induction on the number of blocks of $\tL$s in $t$.
Appending $\tO^r$ to the binary expansion, there is nothing to show by the identity $\cum_j(2t)=\cum_j(t)$.
We append a block of $\tL$s of length $r$:
Using the following (trivial) identity
\[\cum_j\bigl(2^rt+2^r-1\bigr)
=\cum_j(t)+\left(\cum_j\bigl(2^rt+2^r-1\bigr)-\cum_j(t+1)\right)-\left(\cum_j(t)-\cum_j(t+1)\right),\]
and since $\cum_j\bigl(\bigl(2^rt+2^r-1\bigr)+1\bigr)=\cum_j(t+1)$ due to $\cum_j(2t)=\cum_j(t)$, the result follows by Lemma~\ref{lem_aj_diff_bounds}.
\end{proof}
The following lower bound is~\cite[Lemma~3.1]{S2018}, and essentially contained in~\cite{DLP2005}; see also~\cite{EH2018}.
\begin{lemma}\label{lem_a2_lower_bound}
Let $M$ be the number of blocks of $\tL$s in $t$.
Then $\cum_2(t)\geq M$.
\end{lemma}
We prove the following upper bound for $\widetilde\gamma_t(\vartheta)$, using the recurrence~\eqref{eqn_dt_rec} as an essential input.
This proposition is the central property in our proof of the main theorem, showing the crucial uniformity of our approximation.
\begin{proposition}\label{prp_tildegamma_est}
There exists a constant $C$
such that for $\lvert\vartheta\rvert\leq \min\left(M^{-1/6},\tau^{-1}\right)$ we have
\begin{equation*}
\begin{aligned}
\bigl\lvert\widetilde\gamma_t(\vartheta)\bigr\rvert
&\leq CM\vartheta^6,\\
\bigl\lvert \xi_t(\vartheta)\bigr\rvert
&\leq C\vartheta^6,
\end{aligned}
\end{equation*}
where $M$ is the number of blocks of $\tL$s in $t$.
\end{proposition}
\begin{proof}
From \eqref{eqn_gammaprime_def} and \eqref{eqn_tilde_gamma_rec} we see that by construction $\widetilde\gamma_t(\vartheta) = \mathcal{O}(\vartheta^6)$ and $\xi_t(\vartheta) = \mathcal{O}(\vartheta^6)$ as the Taylor coefficients at $\vartheta=0$ of $\gamma_t(\vartheta)$ and $\gamma_t'(\vartheta)$ up to $\vartheta^5$ are the same.
It remains to show that the constants are effective and uniform in $t$.
To begin with, there is a constant $C$ such that~\eqref{eqn_hyp_strenghened} holds for $t\in\{0,1\}$;
a numerical value can be extracted from the first few $\widetilde\gamma_t(\vartheta)$ and $\xi_t(\vartheta)$, which have explicit expansions.

We proceed by induction on the length $L$ of the binary expansion of $t$.
As induction hypothesis, we choose the following strengthened statement:
\begin{equation}\label{eqn_hyp_strenghened}
\begin{array}{c}
\begin{aligned}
\bigl\lvert\widetilde\gamma_t(\vartheta)\bigr\rvert&\leq 2CM\vartheta^6;\\
\bigl\lvert\widetilde\gamma_{t+1}(\vartheta)\bigr\rvert&\leq 2CM\vartheta^6;\\
\bigl\lvert\xi_t(\vartheta)\bigr\rvert&\leq C\vartheta^6
\end{aligned}
\\
\mbox{
\begin{minipage}[t]{0.8\textwidth}
\emph{for all $t$ whose binary expansion has a length bounded by $L$,
and for all real $\vartheta$ satisfying
$\lvert\vartheta\rvert\le1/\tau$ and $\lvert\vartheta\rvert\leq {M}^{-1/6}$, where $M$ is the number of blocks in $t$.}
\end{minipage}
}
\end{array}
\end{equation}
Note that in this proof, and in this proof only, we use the total number of blocks instead of the number of blocks of $\tL$s because this works well with the induction statement. 
The statement of the proposition is not changed by this, since the numbers of blocks of $\tO$s and blocks of $\tL$s differ at most by one.

The statement holds for $t\in\{0,1\}$.
We therefore assume that~\eqref{eqn_hyp_strenghened} holds for all $t$ whose binary expansion has a length strictly less than $L$, where $L\geq 2$.
Our strategy is now to first prove the inequalities for $\widetilde\gamma_t(\vartheta)$ and $\widetilde\gamma_{t+1}(\vartheta)$, and after that the one for $\xi_t(\vartheta)$.
In order to make the interplay between the statements in the induction hypothesis explicit, we rewrite~\eqref{eqn_tilde_gamma_rec} as a matrix recurrence for $t \geq 1$:
\begin{equation*}
\begin{aligned}\left(\begin{matrix}\widetilde\gamma_{2t}(\vartheta)\\\widetilde\gamma_{2t+1}(\vartheta)\end{matrix}\right)&=
A_0\left(\begin{matrix}\widetilde\gamma_t(\vartheta)\\\widetilde\gamma_{t+1}(\vartheta)\end{matrix}\right)+\left(\begin{matrix}0\\\xi_t(\vartheta)\end{matrix}\right) 
&& \text{ with } &
A_0&=\left(\begin{matrix}1&0\\[2mm]\frac{\e(\vartheta)}2&\frac{\e(-\vartheta)}2\end{matrix}\right);
\\[2mm]
\left(\begin{matrix}\widetilde\gamma_{2t+1}(\vartheta)\\\widetilde\gamma_{2t+2}(\vartheta)\end{matrix}\right)&=
A_1\left(\begin{matrix}\widetilde\gamma_t(\vartheta)\\\widetilde\gamma_{t+1}(\vartheta)\end{matrix}\right)+\left(\begin{matrix}\xi_t(\vartheta)\\0\end{matrix}\right)
&& \text{ with } &
A_1&=\left(\begin{matrix}\frac{\e(\vartheta)}2&\frac{\e(-\vartheta)}2\\[2mm]0&1\end{matrix}\right).
\end{aligned}
\end{equation*}
The idea is now to use these relations to reduce the length of $t$.
For this purpose, we regard the run of $\tO$s or $\tL$s at the very right of the binary expansion of~$t$.

First, if we have a run of $\tO$s, we can write $t=2^kt'$, where $t'$ is odd.
Iterating the first matrix equation above, we accumulate powers of $A_0$:
\begin{equation*}
\begin{aligned}\left(\begin{matrix}\widetilde\gamma_{2^kt'}(\vartheta)\\\widetilde\gamma_{2^kt'+1}(\vartheta)\end{matrix}\right)&=
A_0^k\left(\begin{matrix}\widetilde\gamma_{t'}(\vartheta)\\\widetilde\gamma_{t'+1}(\vartheta)\end{matrix}\right)+
\sum_{0\leq j<k}
A_0^{k-1-j}
\left(\begin{matrix}0\\\xi_{2^\ell t'}(\vartheta)\end{matrix}\right)
\\&=
A_0^k\left(\begin{matrix}
\widetilde\gamma_{t'}(\vartheta)\\\widetilde\gamma_{t'+1}(\vartheta)
\end{matrix}\right)+
\left(\begin{matrix}0\\E_0(\vartheta)\end{matrix}\right),
\end{aligned}
\end{equation*}
where, due to $\e(\vartheta)^j = \e(j \vartheta)$, we have
\[E_0(\vartheta)=
\sum_{0\leq j<k}
\frac{\e(-(k-1-j)\vartheta)}{2^{k-1-j}}
\xi_{2^jt'}(\vartheta),
\]
which satisfies
\[\bigl\lvert E_0(\vartheta)\bigr\rvert
\leq 2\max_{0\leq j<k}
\bigl\lvert \xi_{2^j t'}(\vartheta)\bigr\rvert.
\]
Now, the binary length of $2^jt'$ is strictly less than the binary length of $t$, therefore we can use our hypothesis in order to conclude that $\lvert E_0(\vartheta)\rvert\leq 2C\vartheta^6$.
Moreover, the number $M'$ of blocks (of $\tO$s or $\tL$s) in $t'$ is the number $M$ of blocks in $t$ decreased by one (since $t'$ is odd).
By the hypothesis and the fact that $A_0$ has row-sum norm equal to $1$, we obtain
$\lvert\widetilde\gamma_t(\vartheta)\bigr\rvert\leq 2CM\vartheta^6$ and
$\lvert\widetilde\gamma_{t+1}(\vartheta)\bigr\rvert\leq 2CM\vartheta^6$ for $t=2^k t'$.

Second, appending a block of $\tL$s to an even integer $t'$, we obtain from the second matrix equation
\begin{equation*}
\begin{aligned}
\left(\begin{matrix}\widetilde\gamma_{2^kt'+2^k-1}(\vartheta)\\\widetilde\gamma_{2^k(t'+1)}(\vartheta)\end{matrix}\right)&=
A_1^k\left(\begin{matrix}\widetilde\gamma_{t'}(\vartheta)\\\widetilde\gamma_{t'+1}(\vartheta)\end{matrix}\right)+
\left(\begin{matrix}E_1(\vartheta)\\0\end{matrix}\right),
\end{aligned}
\end{equation*}
where
\[E_1(\vartheta)=
\sum_{0\leq j<k}
\frac{\e(-(k-1-j)\vartheta)}{2^{k-1-j}}
\xi_{2^jt'+2^j-1}(\vartheta)
\]
satisfies
\[\bigl\lvert E_1(\vartheta)\bigr\rvert
\leq 2\max_{0\leq j<k}
\bigl\lvert \xi_{2^j t'+2^j-1}(\vartheta)\bigr\rvert.
\]
As above, we have by our induction hypothesis $E_1(\vartheta)\leq 2C\vartheta^6$.
Then, since the integer $t'$ has one block less than $t$ and 
since $A_1$ has row-sum norm equal to $1$, we can use our induction hypothesis~\eqref{eqn_hyp_strenghened} and get
$\lvert\widetilde\gamma_t(\vartheta)\bigr\rvert\leq 2CM\vartheta^6$ and
$\lvert\widetilde\gamma_{t+1}(\vartheta)\bigr\rvert\leq 2CM\vartheta^6$ for $t=2^k t' + 2^k-1$.

It remains to consider the inequality for $\xi_t(\vartheta)$.
We start by dividing Equation~\eqref{eqn_xi_def} by $\gammap_t(\vartheta)$.
This gives
\begin{equation}\label{eqn_xi_factoring}
\begin{aligned}
\hspace{6em}&\hspace{-6em}
\frac{\xi_t(\vartheta)}{\gammap_t(\vartheta)} =
\frac{\e(\vartheta)}2
+\frac{\e(-\vartheta)}2 \exp\biggl(\sum_{2\leq j\leq 5}\frac{\cum_j(t+1)-\cum_j(t)}{j!}(i\tau\vartheta)^j\biggr)
\\&
-\exp\biggl(\sum_{2\leq j\leq 5}\frac{\cum_j(2t+1)-\cum_j(t)}{j!}(i\tau\vartheta)^j\biggr)
\end{aligned}\end{equation}
As observed before, we have $\xi_t(\vartheta) = \mathcal{O}(\vartheta^6)$ and consequently, dividing by the power series $\gammap_t(\vartheta)=1+\mathcal O(\vartheta^2)$, we see that the series of the right hand side also belongs to $\mathcal{O}(\vartheta^6)$.
Next, 
we get by the triangle inequality and the induction hypothesis
\begin{align*}
\bigl\lvert \gammap_t(\vartheta)\bigr\rvert
&\leq
\bigl\lvert \gamma_t(\vartheta)\bigr\rvert
+
\bigl\lvert \widetilde\gamma_t(\vartheta)\bigr\rvert
\leq
1+2CM\vartheta^6
\end{align*}
and since $\vartheta\leq M^{-1/6}$,
we obtain
\[\bigl\lvert \gammap_t(\vartheta)\bigr\rvert=\mathcal O(1).\]

Now we turn our attention to the right hand side of~\eqref{eqn_xi_factoring}, where we will treat each summand separately.
The first term $\e(\vartheta)/2$ has $(i\,\tau)^k/(2\cdot k!)$ as coefficients;
since $\tau\vartheta\leq 1$,
the contribution of the coefficients for $k\geq 6$ is bounded by
\[\frac 12\sum_{k\geq 6}\frac{(\tau\vartheta)^k}{k!}
\leq \frac 12(\tau\vartheta)^6(e-163/60)<\frac 1{1234}(\tau\vartheta)^6.
\]

Next, we want to show that the contribution of the second term (i.e., the product of two exponentials) and the third term are each bounded by $C(\tau\vartheta)^6$.
By Lemma~\ref{lem_aj_diff_bounds}, an upper bound for the coefficients of the second term is given by the coefficients of
\[f(\vartheta)=\exp\left(2\bigl((\tau\vartheta)+\cdots+(\tau\vartheta)^5\bigr)\right).\]
Clearly, the term $\vartheta^k$ in the $j$-fold product $(\vartheta+\vartheta^2+\cdots+\vartheta^5)^j$ appears at most $5^j$ times, but only for $j\geq k/5$.
Therefore the coefficient $[\vartheta^k]f(\vartheta)$ is bounded by
\[  \tau^k\sum_{k/5\leq j\leq k}2^j \frac{5^j}{j!}\leq
    \tau^k\sum_{j\geq k/5}\frac{10^j}{j!}.   \]
Consequently, as we only need to consider coefficients of $\vartheta^k$ with $k \geq 6$, and since $\lvert \tau\vartheta\rvert\leq 1$, we get
\begin{align*}
\sum_{k\geq 6}\vartheta^k\bigl[\vartheta^k\bigr]f(\vartheta)
\leq (\tau\vartheta)^6\sum_{k\geq 6}\sum_{j\geq k/5}\frac{10^j}{j!}
\leq 5(\tau\vartheta)^6\sum_{j\geq 1}\frac{10^jj}{j!}
\leq C'\vartheta^6
\end{align*}
for some absolute constant $C'$.
The same holds for the third exponential in~\eqref{eqn_xi_factoring}, as $\lvert\cum_j(2t+1)-\cum_j(t)\rvert=\lvert\cum_j(2t+1)-\cum_j(2t)\rvert\leq 240$.
Collecting these results we get an absolute and effective constant $C$ such that $\lvert\xi_t(\vartheta)\rvert\leq C\vartheta^6$ 
as long as $\vartheta\leq M^{-1/6}$ and $\lvert\tau\vartheta\rvert\leq 1$.
\end{proof}

\pagebreak

\subsection{An integral representation of \texorpdfstring{$c_t$}{c\_t}}
\label{sec:intct}
We use the following representation of the values $c_t$.
\begin{proposition}[{\cite[Proposition~2.1]{S2020}}]\label{prp_ct_integral}
Let $t\geq 0$.
We have
\begin{equation}\label{eqn_ct_rep}
  c_t = \frac 12 + \frac{\delta(0,t)}2 +
  \frac 12\int_{-1/2}^{1/2}
  \imagpart \gamma_t(\vartheta)\cot(\pi \vartheta)\,\mathrm d\vartheta,
\end{equation}
where the integrand is a bounded, continuous function.
\end{proposition}
We split the integral at the points $\pm\vartheta_0$, where $\vartheta_0=M^{-1/2}R$.
Here $M$ is the number of blocks of $\tL$s in $t$ and $R$ is a small parameter to be chosen in a moment.
For now, we assume that
\begin{equation}\label{eqn_technical}
\begin{aligned}
8&\leq R\leq M^{1/3}\quad\mbox{and}\quad
\vartheta_0\le1/\tau
\end{aligned}
\end{equation}
for technical reasons as, among others, we need to apply Proposition~\ref{prp_tildegamma_est}.
Note that under these hypotheses,
\[\vartheta_0\leq M^{-1/6},\]
so that the proposition will be applicable.
We will choose $R=\log M$; then~\eqref{eqn_technical} will be satisfied for large $M$.
The tails of the above integral will be estimated using the following lemma.
\begin{lemma}[{\cite[Lemma~2.7]{S2020}}]\label{lem_gamma_tail}
Assume that $t\geq 1$ has at least $M=2M'+1$ blocks of $\tL$s.
Then
\begin{equation*}
\left\lvert\gamma_t(\vartheta)\right\rvert\leq
\left(1-\frac{\vartheta^2}2\right)^{M'}
\leq
\exp\left(-\frac{M'\vartheta^2}2\right)
\leq 2\exp\left(-\frac{M\vartheta^2}4\right)
\end{equation*}
for $\lvert\vartheta\rvert\leq1/2$.
\end{lemma}
We have $\cot(x)=1/x+\mathcal O(1)$ for $x\leq 1/2$.
The contribution of the tail can therefore be bounded by
\[\int_{M^{-1/2}R}^{1/2}\exp\left(-\frac{M\vartheta^2}4\right)\cot(\pi\vartheta)\,\mathrm d\vartheta
\leq 
\frac 1\pi
I+
\mathcal O\left(J\right),
\]
where 
\[I=\int_{M^{-1/2}R}^{\infty}\exp\left(-\frac{M\vartheta^2}4\right)\frac{\mathrm d\vartheta}{\vartheta}\]
and
\[J=\int_{M^{-1/2}R}^{\infty}\exp\left(-\frac{M\vartheta^2}4\right)\,\mathrm d\vartheta.
\]

The integral $J$ is bounded by
\[\mathcal O\left(\exp\bigl(-M(M^{-1/2}R)^2/4\bigr)\right)=\mathcal O\left(\exp\bigl(-R^2/4\bigr)\right).\]

In order to estimate $I$, we write
\[I\leq \sum_{j\geq 0}\int_{2^j\vartheta_0}^{2^{j+1}\vartheta_0}\exp\left(-\frac{M\vartheta^2}{4}\right)\frac{\mathrm d\vartheta}{2^j\vartheta_0}
\leq
\sum_{j\geq 0}\exp\left(-\frac{4^jR^2}{4}\right).
\]
Using the hypothesis $R\geq 1$, this is easily shown to be bounded by $\mathcal O\left(\exp\bigl(-R^2/4\bigr)\right)$ by a geometric series.
For $\lvert \vartheta\rvert\leq \vartheta_0$, we replace $\gamma_t(\vartheta)$ by $\gammap_t(\vartheta)$ in the integral in~\eqref{eqn_ct_rep}, using Proposition~\ref{prp_tildegamma_est}.
Noting the hypotheses~\eqref{eqn_technical}, we obtain
$\lvert \gamma_t(\vartheta)-\gammap_t(\vartheta)\rvert\ll M \lvert\vartheta\rvert^6$,
where $M$ is the number of blocks in $t$.
Therefore
\begin{align}
\notag
\hspace{3em}&\hspace{-3em}
\int_{-1/2}^{1/2}\imagpart \gamma_t(\vartheta)\cot(\pi\vartheta)\,\mathrm d\vartheta
=
\int_{-\vartheta_0}^{\vartheta_0}\imagpart \gamma_t(\vartheta)\cot(\pi\vartheta)\,\mathrm d\vartheta+\mathcal O\bigl(\exp\bigl(-R^2/4\bigr)\bigr)
\\
\label{eqn_imagpart}
&=
\int_{-\vartheta_0}^{\vartheta_0}\imagpart \gammap_t(\vartheta)\cot(\pi\vartheta)\,\mathrm d\vartheta+
\mathcal O\left(
M\int_{0}^{\vartheta_0}\vartheta^{5}\,\mathrm d\vartheta
\right)
+
\mathcal O\bigl(\exp\bigl(-R^2/4\bigr)\bigr)
\\
\notag
&=
\int_{-\vartheta_0}^{\vartheta_0}\imagpart \gammap_t(\vartheta)\cot(\pi\vartheta)\,\mathrm d\vartheta+
\mathcal O(E),
\end{align}
where, due to $\vartheta_0 = M^{-1/2} R$, we have
\[E=M^{-2}R^6+\exp\bigl(-R^2/4\bigr).\]
Similarly, combining~\eqref{eqn_gamma_integral} with the above reasoning, we get 
\begin{equation}\label{eqn_delta_0t}
\delta(0,t)=\int_{-\vartheta_0}^{\vartheta_0}
\realpart \gammap_t(\vartheta)\,\mathrm d\vartheta+\mathcal O(E).
\end{equation}
Next we return to the definition of $\gammap_t(\vartheta)$ from~\eqref{eqn_gammaprime_def}.
By the Taylor expansion of $\exp$,
using Corollary \ref{cor_aj_bounds}, we have for $\lvert\vartheta\rvert\leq \vartheta_0$
\begin{align*}
\gammap_t(\vartheta)
&=
\exp\left(-\cum_2(t)\frac{(\tau\vartheta)^2}2\right)
\times
\Bigl(1
+\frac{\cum_3(t)}{6}(i\tau\vartheta)^3
+\frac{\cum_4(t)}{24}(i\tau\vartheta)^4
+\frac{\cum_5(t)}{120}(i\tau\vartheta)^5
\Bigr.\\&\Bigl.
+\frac1{72}\cum_3(t)^2(i\tau\vartheta)^6
+\frac1{144}\cum_3(t)\cum_4(t)(i\tau\vartheta)^7
+\frac1{1296}\cum_3(t)^3(i\tau\vartheta)^9
\Bigr)
\\&
+\mathcal O\bigl(M^2\vartheta^8+M^3\vartheta^{10}\bigr)
+i\mathcal O\bigl(M^2\vartheta^9+M^3\vartheta^{11}\bigr),
\end{align*}
where both error terms are real.
We note that $\cot(\pi \vartheta)=2/(\tau \vartheta)-\tau\vartheta/6+\mathcal O(\vartheta^3)$ for $\lvert\vartheta\rvert\leq 1/2$.
Splitting into real and imaginary summands, of which there are three and four, respectively, we obtain by~\eqref{eqn_imagpart} and~\eqref{eqn_delta_0t}
\begin{align*}c_t&=\frac 12+\frac 12\int_{-\vartheta_0}^{\vartheta_0}
\exp\left(-\cum_2(t)\frac{(\tau\vartheta)^2}2\right)
\biggl(
1+\frac{\cum_4(t)}{24}(\tau\vartheta)^4
-\frac1{72}\cum_3(t)^2(\tau\vartheta)^6
\Bigr.\\&\Bigl.
\quad+\biggl(
-\frac16\cum_3(t)(\tau\vartheta)^3
+\frac1{120}\cum_5(t)(\tau\vartheta)^5
-\frac1{144}
\cum_3(t)\cum_4(t)(\tau\vartheta)^7
\\&\quad
+
\frac1{1296}
\cum_3(t)^3
(\tau\vartheta)^9\biggr)\cot(\pi\vartheta)
\biggr)
\,\mathrm d\vartheta
+\mathcal O\bigl(E+E_2\bigr)
\\&=\frac 12+
\frac 12\int_{-\vartheta_0}^{\vartheta_0}
\exp\left(-\cum_2(t)\frac{(\tau\vartheta)^2}2\right)
\biggl(
1+\frac{\cum_4(t)}{24}(\tau\vartheta)^4
-\frac{\cum_3(t)^2}{72}(\tau\vartheta)^6
-\frac{\cum_3(t)}{3}(\tau\vartheta)^2
\\&\quad
+\frac{\cum_5(t)}{60}(\tau\vartheta)^4
-\frac{\cum_3(t)\cum_4(t)}{72}(\tau\vartheta)^6
+
\frac{\cum_3(t)^3}{648}(\tau\vartheta)^8
+\frac{\cum_3(t)}{36}(\tau\vartheta)^4\biggr)
\,\mathrm d\vartheta
+\mathcal O\bigl(E+E_2\bigr),
\end{align*}
where
\[E_2=\int_{-\vartheta_0}^{\vartheta_0}
\left(M\vartheta^6+M^2\vartheta^8+M^3\vartheta^{10}\right)\,\mathrm d\vartheta
\ll M^{-5/2}R^{11}.
\]
We extend the integration limits again, introducing an error
\[E_3\ll\int_{M^{-1/2}R}^\infty \exp\left(-\cum_2(t)\frac{\vartheta^2}2\right)\bigl(1+M\vartheta^2+M\vartheta^4+M^2\vartheta^6+M^3\vartheta^8\bigr).\]
In order to estimate this, we use the following lemma.
\begin{lemma}\label{lem_gaussian_integrals}
For real numbers $a>0$ and $\delta\geq 0$, and integers $j\geq 0$, we define
\[I_j=\int_\delta^\infty x^j \exp(-ax^2).\]
Then 
\begin{align*}
I_2&\ll \frac\delta a\exp\bigl(-a\delta^2\bigr),\\
I_4&\ll \left(\frac{\delta^3}a+\frac {\delta}{a^2}\right)\exp\bigl(-a\delta^2\bigr),\\
I_6&\ll \left(\frac{\delta^5}a+\frac{\delta^3}{a^2}+\frac{\delta}{a^3}\right)\exp\bigl(-a\delta^2\bigr),\\
I_8&\ll \left(\frac{\delta^7}a+\frac{\delta^5}{a^2}+\frac{\delta^3}{a^3}+\frac{\delta}{a^4}\right)\exp\bigl(-a\delta^2\bigr).
\end{align*}
\end{lemma}
\begin{proof}
We have
\[\frac{\partial}{\partial x}
x^m \exp\bigl(-ax^2\bigr)
=
\bigl(mx^{m-1}-2ax^{m+1}\bigr)
\exp\bigl(-ax^2\bigr),
\]
therefore
\[I_{m+1}
=-\frac {x^m}{2a}\exp(-ax^2)\Big\vert_{\delta}^\infty
+\frac{m}{2a}I_{m-1}.
\]
Noting that $I_0\ll \exp\bigl(-a\delta^2\bigr)$,
we obtain the above estimates by recurrence.
\end{proof}
We insert $a=\cum_2(t)/2$ and $\delta=\vartheta_0$.
By Lemma~\ref{lem_a2_lower_bound}
we have $a\geq M/2>0$,
and by our hypothesis~\eqref{eqn_technical} we have $R\leq M^{1/6}$,
which implies in particular that $\delta=M^{-1/2}R\leq1$.
By these estimates and Lemma~\ref{lem_gaussian_integrals}, 
we obtain
\begin{align*}
E_3&\ll
\left(1+M^{-1/2}R + M^{-3/2}R^7 \right)
\exp\left(-\cum_2(t)(M^{-1/2}R)^2/2\right)
\\&
\ll
\exp\left(-R^2/2\right)\ll E.
\end{align*}

Substituting $\tau\vartheta$ by $\vartheta$, we obtain
\begin{align*}
c_t&=\frac 12+\frac 1{2\tau}\int_{-\infty}^{\infty}
\exp\biggl(-\cum_2(t)\frac{\vartheta^2}2\biggr)
\biggl(1-\frac{\cum_3(t)}3\vartheta^2
+\biggl(\frac{\cum_3(t)}{36}
+\frac{\cum_4(t)}{24}
+\frac{\cum_5(t)}{60}\biggr)\vartheta^4
\\&
+\biggl(-\frac{\cum_3(t)}{72}-\frac{\cum_4(t)}{72}\biggr)\cum_3(t)\vartheta^6
+\frac{\cum_3(t)^3}{648}\vartheta^8
\biggr)
\,\mathrm d\vartheta+\mathcal O\bigl(E+E_2\bigr).
\end{align*}
Inserting standard Gaussian integrals,
it follows that
\begin{equation}\label{eqn_ct_Aj_reduction}
\begin{aligned}
c_t&=
\frac 12+\frac{\sqrt{2}}{4\sqrt{\pi}}
\biggl(\cum_2(t)^{-1/2}
-\frac{\cum_2(t)^{-3/2}\cum_3(t)}3
\\&
+3\cum_2(t)^{-5/2}
\biggl(
\frac{\cum_3(t)}{36}
+\frac{\cum_4(t)}{24}
+\frac{\cum_5(t)}{60}
\biggr)
\\&
+15\cum_2(t)^{-7/2}
\biggl(
-\frac{\cum_3(t)}{72}-\frac{\cum_4(t)}{72}
\biggr)\cum_3(t)
+105\cum_2(t)^{-9/2}
\frac{\cum_3(t)^3}{648}
\biggr)
\\&+\mathcal O\left(M^{-2}R^{11}+\exp\bigl(-R^2/4\bigr)\right)
\end{aligned}
\end{equation}
under the hypotheses that $8\leq R\leq M^{1/6}$ and 
$M^{-1/2}R\leq 1/\tau$,
where $M$ is the number of blocks of $\tL$s in $t$.
The multiplicative constant in the error term is absolute, as customary in this paper.
In order to simplify the error term, we choose
\begin{equation}\label{eqn_R_choice}
R=\log M.
\end{equation}
Using the hypothesis $R\geq 8$,
we have $\exp\bigl(-R^2/4\bigr)\leq M^{-2}$.
Then, since $\cum_2(t)\geq 0$ for all $t$, we see that for $c_t>1/2$ it is sufficient to prove
\begin{equation*}
v(t)\geq 0,
\end{equation*}
where
\begin{equation}\label{eqn_v_def}
\begin{aligned}
v(t)&=
\cum_2(t)^4
-
\cum_2(t)^3
\frac{\cum_3(t)}3
+\cum_2(t)^2
\biggl(
\frac{\cum_3(t)}{12}
+\frac{\cum_4(t)}{8}
+\frac{\cum_5(t)}{20}
\biggr)
\\&
+5\cum_2(t)
\biggl(
-\frac{\cum_3(t)}{24}-\frac{\cum_4(t)}{24}
\biggr)\cum_3(t)
+35\frac{\cum_3(t)^3}{216}
-C\cum_2(t)^{5/2}R^{11}.
\end{aligned}
\end{equation}
and $C$ is large enough such that the error term in~\eqref{eqn_ct_Aj_reduction} is strictly dominated by $C\cum_2(t)^{5/2}M^{-2}R^{11}$.
Usually the first term is the dominant one;
the critical cases occur when the first two terms in~\eqref{eqn_v_def} almost cancel.
We couple these terms and write
\[D=D(t)=\cum_2(t)-\frac{\cum_3(t)}3.\]
Let us rewrite the expression for $v(t)$, eliminating $\cum_3(t)$.
Clearly, we have $\cum_3(t)^2=9\cum_2(t)^2-18D\cum_2(t)+9D^2$
and
$\cum_3(t)^3=27\cum_3(t)^3-81D\cum_3(t)^2+81D^2\cum_3(t)-27D^3$.
Omitting the argument $t$ of the functions $\cum_j$ for brevity, we obtain
\begin{equation}\label{eqn_vt_long}
\begin{aligned}
v(t)&=D\cum_2^3
+\frac 14\cum_2^3-\frac14D\cum_2^2
+\frac18\cum_2^2\cum_4
+\frac1{20}\cum_2^2\cum_5
\\&\quad
-\frac{15}8 \cum_2^3
+\frac{15}4D\cum_2^2
-\frac{15}8D^2\cum_2
-\frac 58\cum_2^2\cum_4
+\frac58D\cum_2\cum_4
\\&\quad
+\frac{35}{8}
\biggl(
\cum_2^3-3D\cum_2^2+3D^2\cum_2-D^3
\biggr)
-C\cum_2^{5/2}R^{11}
\\&=
\left(D+\frac{11}{4}\right)\cum_2^3
-\frac12\cum_2^2\cum_4
+\frac1{20}\cum_2^2\cum_5
-\frac{77}8D\cum_2^2
\\&\quad
+\frac58D\cum_2\cum_4
+\frac{45}4D^2\cum_2
-\frac{35}8D^3
-C\cum_2^{5/2}R^{11}.
\end{aligned}
\end{equation}

We distinguish between small and large values of $D$.
Note that $|\cum_j|\leq CM$, $D\leq CM$ for some absolute constant $C$ (expressed in Corollary~\ref{cor_aj_bounds}),
moreover $\cum_2\geq M$ (Proposition~\ref{lem_a2_lower_bound}) and $R=\log M$.
Thus, we have $|\cum_j| \leq C \cum_2$ and $D \leq C \cum_2$.
Therefore there exists an absolute constant $D_0$ (which could be made explicit easily) such that 
\begin{equation}\label{eqn_vt_easy_bound}
v(t)\geq \bigl(D(t)-D_0\bigr)\cum_2(t)^3
\end{equation}
for all $t\geq 1$.
Clearly this implies $v(t)\geq 0$ for all $t$ such that $D(t)\geq D_0$.
We have therefore proved the following result.
\begin{lemma}\label{lem_ct_uncritical}
There exists a constant $D_0$ such that, if 
$\cum_2(t)-\cum_3(t)/3\geq D_0$, then $c_t>1/2$.
\end{lemma}
The remainder of the proof of Theorem~\ref{thm_main} is concerned with the case $D(t)<D_0$.
As $D_0$ is an absolute constant, independent of $t$ and $M$, we see that $D(t)/\cum_2(t)^{\lambda}$ with $\lambda>0$ becomes arbitrarily small when the number of blocks in $t$ increase. 
Thus, we obtain from~\eqref{eqn_vt_long} the following statement:
for all $\varepsilon>0$ there is an $M_0$ such that for $M\geq M_0$ we have
\begin{align}\label{eqn_11_4}
v(t)&\geq
\left(D+\frac{11}4-\varepsilon\right)\cum_2^3
-\frac12\cum_2^2\cum_4
+\frac1{20}\cum_2^2\cum_5.
\end{align}
We proceed by taking a closer look at the values $D(t)$.
We have
\begin{align*}
D(2t+1)&=\frac{\cum_2(t)+\cum_2(t+1)}2-\frac{\cum_3(t)+\cum_3(t+1)}6-\frac{\cum_2(t)-\cum_2(t+1)}2+1,
\end{align*}
therefore
\begin{equation}\label{eqn_D_rec}
D(2t)=D(t)\quad\textrm{and}\quad D(2t+1)=
\frac{D(t)+D(t+1)}2
+
\frac{\cum_2(t+1)-\cum_2(t)}2
+1.
\end{equation}
By~\eqref{eqn_fubini}, we have $D(1)=D(2)=4$,
moreover the term $(\cum_2(t+1)-\cum_2(t))/2+1$ is nonnegative by Lemma~\ref{lem_aj_diff_bounds}.
This implies 
\begin{equation}\label{eqn_D_2}
D(t)\geq 4.
\end{equation}
Choosing $\varepsilon=1/8$ in~\eqref{eqn_11_4},
we see that it remains to show that
\begin{align}\label{eqn_lincomb_sufficient}
53\cum_2
-4\cum_4
+\frac25\cum_5>0
\end{align}
if $t$ contains many blocks, and $D(t)$ is bounded by some absolute constant $D_0$.

This is done in two steps:
first, we determine the structure of the exceptional set of integers~$t$ such that $D(t)$ is bounded.
We will see that such an integer has few blocks of $\tO$s of length $\geq 2$, and few blocks of $\tL$s of bounded length.
As a second step, we prove lower bounds for the numbers $-\cum_4(t)$ and $\cum_5(t)$, if $t$ is contained in this exceptional set.

\subsection{Determining the exceptional set}
\label{sec:determinintexceptionalset}
We define the exceptional set 
\begin{equation*}
	\{t : D(t)<D_0 \}, 
\end{equation*}
where $D_0$ is the constant from Lemma~\ref{lem_ct_uncritical}.
In this section we will derive some structural properties of its elements.

We begin with investigating the effect of appending a block of the form $\tO\tL^k$.
\begin{lemma}\label{lem_OLL}
For $t\geq 0$ and $k\geq 0$ we have
\begin{equation}\label{eqn_OLL_explicit}
\cum_2(2^{k+1}t+2^k-1)=
\frac{(2^k+1)\cum_2(t)}{2^{k+1}}+\frac{(2^k-1)\cum_2(t+1)}{2^{k+1}}+\frac {3\left(2^k-1\right)}{2^k},
\end{equation}
\begin{equation}\label{eqn_DOLL}
\begin{aligned}
D(2^{k+1}t+2^k-1)
&=\frac{2^k+1}{2^{k+1}}D(t)+\frac{2^k-1}{2^{k+1}}D(t+1)
\\&+\left(\frac 12+\frac{k-1}{2^{k+1}}\right)
\bigl(\cum_2(t+1)-\cum_2(t)\bigr)+
1+\frac{3k-1}{2^k}.
\end{aligned}
\end{equation}
\end{lemma}
\begin{proof}
The proof of the first part is easy, using induction and the recurrence~\eqref{eqn_dt_rec}.

We continue with the second part.
The statement is trivial for $k=0$ and for $k=1$ it follows from~\eqref{eqn_OLL_explicit}.
We use the abbreviations
$\rho_k=1/2+(k-1)/2^{k+1}$ and $\sigma_k=1+(3k-1)/2^k$.
For $k\geq 1$ we have by induction, using~\eqref{eqn_D_rec} and~\eqref{eqn_OLL_explicit},
\begin{align*}
\hspace{1em}&\hspace{-1em}
D(2^{k+2}t+2^{k+1}-1)
=\frac{D(2^{k+1}t+2^k-1)+D(2t+1)}2
\\&\quad+\frac{\cum_2(2t+1)-\cum_2(2^{k+1}t+2^k-1)}2+1
\\&=
\frac{2^k+1}{2^{k+2}}D(t)+\frac{2^k-1}{2^{k+2}}D(t+1)
+\frac{\rho_k}2\bigl(\cum_2(t+1)-\cum_2(t)\bigr)+\frac{\sigma_k}2
\\&\quad+
\frac{D(t)+D(t+1)}4+\frac{\cum_2(t+1)-\cum_2(t)}4+\frac 12
+\frac{\cum_2(t)+\cum_2(t+1)}4+\frac 12
\\&\quad-\frac 12\left(\frac{2^k+1}{2^{k+1}}\cum_2(t)+\frac{2^k-1}{2^{k+1}}\cum_2(t+1)+3\frac{2^k-1}{2^k}\right)
+1
\\&=\frac{2^{k+1}+1}{2^{k+2}}D(t)+\frac{2^{k+1}-1}{2^{k+2}}D(t+1)+\left(\frac {\rho_k}2+\frac 14+\frac1{2^{k+2}}\right)\bigl(\cum_2(t+1)-\cum_2(t)\bigr)
\\&\quad+\frac{\sigma_k}2+\frac 12+\frac3{2^{k+1}},
\end{align*}
which implies the statement.
\end{proof}
We obtain the following corollary.
\begin{corollary}\label{cor_D_append_OLL}
For all $t\geq 0$ and $k\geq 1$ we have
\[
D\bigl(2^{k+1}t+2^k-1\bigr)\geq \min\bigl(D(t),D(t+1)\bigr)+\frac k{2^{k-1}}.
\]
\end{corollary}
\begin{proof}
Set $\alpha=\bigl(2^k+1\bigr)/2^{k+1}$ and $\beta=\bigl(2^k-1\bigr)/2^{k+1}$.
By the bound $\lvert \cum_2(t+1)-\cum_2(t)\rvert\leq 2$ from Lemma~\ref{lem_aj_diff_bounds}, it follows from Equation~\eqref{eqn_DOLL} that
\begin{align*}
D\bigl(2^{k+1}t+2^k-1\bigr)
&\geq
\alpha D(t)+\beta D(t+1)
+
\frac 12
\bigl(\cum_2(t+1)-\cum_2(t)+2\bigr)
+\frac k{2^{k-1}}
\\&\geq
\min\bigl(D(t),D(t+1)\bigr)+\frac k{2^{k-1}}.
\qedhere
\end{align*}
\end{proof}
We can now extract the contribution to the value of $D$ of a block of the form $\tO\tL^k\tO$.
For this, we use the notation
\[m(t)=\min\bigl(D(t),D(t+1)\bigr).\]
This notation is introduced in order to obtain the following \emph{monotonicity  property}:
by the recurrence~\eqref{eqn_D_rec} and the nonnegativity of $a(t)=\bigl(\cum_2(t+1)-\cum_2(t)\bigr)/2+1$ we have
\begin{equation}\label{eqn_m_monotonicity}
\begin{aligned}
\min\bigl(m(2t),m(2t+1)\bigr)
&=\min\left(D(t),\frac{D(t)+D(t+1)}2+a(t),D(t+1)\right)
\\&\geq \min\bigl(D(t),D(t+1)\bigr)=m(t)
\end{aligned}
\end{equation}
Note also $m(t)\geq 4$ by~\eqref{eqn_D_2}.
These properties will be used in an essential way in the important Corollary~\ref{cor_exceptions} below, where an induction along the binary expansion of $t$ is used.
\begin{corollary}\label{cor_OLLO}
For all $t\ge0$ and $k\ge1$ we have
\[m(2^{k+2}t+2^{k+1}-2)\ge m(t)+\frac k{2^k}.\]
\end{corollary}
\begin{proof}
We have $D(2^{k+2}t+2^{k+1}-2)=D(2^{k+1}t+2^k-1)$, and by Corollary~\ref{cor_D_append_OLL} this is bounded below by 
$m(t)
+\frac k{2^{k-1}}$.
Also, $D(2^{k+2}t+2^{k+1}-2+1)=D(2^{k+2}t+2^{k+1}-1)\geq 
m(t)+\frac {k+1}{2^k}$ and clearly,
$\min\bigl(k/2^{k-1},(k+1)/2^k\bigr)\geq k/2^k$.
\end{proof}

Moreover, we want to find the contribution of a block of $\tO$s of length $\geq 2$.
For this, we append $\tO\tO\tL$ and look what happens:
note that
\begin{align*}
\cum_2(4t+1)&=\frac{3\cum_2(t)}4+\frac{\cum_2(t+1)}4+\frac32,\\
D(4t+1)&=\frac{3D(t)}4+\frac{D(t+1)}4+\frac{\cum_2(t+1)-\cum_2(t)}2+2
\end{align*}
by~\eqref{eqn_OLL_explicit} and~\eqref{eqn_DOLL}. Therefore, by the recurrence~\eqref{eqn_D_rec}, we obtain
\begin{align*}
D(8t+1)
&=\frac{D(t)+D(4t+1)}2+\frac{\cum_2(4t+1)-\cum_2(t)}2+1
\\&=\frac78D(t)+\frac18D(t+1)+\frac 38\bigl(\cum_2(t+1)-\cum_2(t)\bigr)+\frac{11}4.
\end{align*}
These formulas together with
$D(8t+2)=D(4t+1)$ and $\lvert \cum_2(t+1)-\cum_2(t)\rvert\leq 2$ show that 
\begin{equation}\label{eqn_OOL}
m(8t+1)\geq m(t)+1.
\end{equation}
\begin{corollary}\label{cor_exceptions}
Assume that $k\geq 2$ and $t\geq 1$ are integers.
Let $K$ be the number of inner blocks of $\tO$s of length at least two
in the binary expansion of $t$, and $L$ be the number of blocks of $\tL$s of length $\leq k$. Then 
\[    m(t)\geq 4+K+\max\left(0,\left\lfloor\frac{L-2K-1}2\right\rfloor\right)\frac k{2^k}.    \]

In particular, for all integers $D_0\geq 2$ and $k\geq 2$, there exists a bound $B=B(D_0,k)$ with the following property:
for all integers $t\geq 1$ such that $D(t)\leq D_0$, the number of inner blocks of $\tO$s of length $\geq 2$ in $t$ and the number of blocks of $\tL$s of length $\leq k$ in $t$ are bounded by $B$.
\end{corollary}

\begin{proof}
We are going to apply~\eqref{eqn_OOL} $K$ times and Corollary~\ref{cor_OLLO} $\lfloor (L-2K-1)/2\rfloor$ times, using the monotonicity of $m$ expressed in~\eqref{eqn_m_monotonicity} in an essential way.
We proceed by induction along the binary expansion of $t$, beginning at the most significant digit.
The constant $4$ is explained by the starting value $m(1)=\min(D(1),D(2))=4$.
Each inner block of $\tO$s of length $\geq 2$ (bordered by $\tL$s on both sides)
corresponds to a factor $\tO\tO\tL$ in the binary expansion: we simply choose the block of length three starting at the second zero from the right.
Therefore~\eqref{eqn_OOL} explains the contribution $K$.
For the application of Corollary~\ref{cor_OLLO} we need a block of the form $\tO\tL^r\tO$ with $r\geq 1$, but we cannot guarantee that the adjacent blocks of $\tO$s have not already been used for~\eqref{eqn_OOL}.
Therefore each of the $K$ inner blocks of $\tO$s of length $\geq 2$ renders the two adjacent blocks of $\tL$s unusable for the application of Corollary~\ref{cor_OLLO}.
Out of the remaining blocks of $\tL$s of length $\le k$, we can only use each second block, and the first and the last blocks of $\tL$s are excluded also.
That is, if $L-2K\in\{3,4\}$, we can apply Corollary~\ref{cor_OLLO} once, for $L-2K\in\{5,6\}$ twice, and so on.
Finally, we note that $k/2^k$ is nonincreasing.
This explains the last summand. 
\end{proof}

In the following, we will only use the ``in particular''-statement of Corollary~\ref{cor_exceptions}.

\subsection{Bounds for \texorpdfstring{$\cum_4$}{kappa4} and \texorpdfstring{$\cum_5$}{kappa5}}
\label{sec:boundsK4K5}
\begin{lemma}\label{lem_cum4_upper_bound}
Assume that $t$ contains $M$ blocks of $\tL$s.
Then
\[\cum_4(t)\leq 26(M+1).\] 
\end{lemma}
\begin{proof}
Recall that $\cum_4(1)=26$ by~\eqref{eqn_fubini}.
Using~\eqref{eqn_dt_rec} and the estimates from Lemma~\ref{lem_aj_diff_bounds} we get 
\[\cum_4(2t+1)\leq \frac{\cum_4(t)+\cum_4(t+1)}2+13.\]
Using the geometric series, this implies
\begin{equation}\label{eqn_A4_LL}
\cum_4\bigl(2^kt+2^k-1\bigr)
\leq
\frac{\cum_4(t)}{2^k}
+\frac{\bigl(2^k-1\bigr)\cum_4(t+1)}{2^k}+26.
\end{equation}
The statement for $M=1$ easily follows.
We also study $t'=2^kt+1$:
In this case, we have
\begin{equation}\label{eqn_A4_OLL}
\cum_4\bigl(2^kt+1\bigr)
\geq 
\frac{\bigl(2^k-1\bigr)\cum_4(t)}{2^k}
+\frac{\cum_4(t+1)}{2^k}+13
\end{equation}
by induction.
We consider the values $n(t)=\min(\cum_4(t),\cum_4(t+1))$
and prove the stronger statement that $n(t)\geq 26(M+1)$ by induction.
We append a block $\tL^k$ to $t$ and obtain $t'=2^kt+2^k-1$.
Then
\[\cum_4(t')\leq \frac{\cum_4(t)}{2^k}+\frac{\bigl(2^k-1\bigr)\cum_4(t+1)}{2^k}+26
\leq \min(\cum_4(t),\cum_4(t+1))+26=n(t)+26,\]
and $\cum_4(t'+1)=\cum_4(t+1)$.
Analogously, we append $\tO^k$ to $t$ and obtain $t'=2^kt$.
Clearly, $\cum_4(t')=\cum_4(t)$, and
\[\cum_4(t'+1)\leq \frac{\bigl(2^k-1\bigr)\cum_4(t)}{2^k}+\frac{\cum_4(t+1)}{2^k}+26
\geq n(t)+26.\]
This implies the statement.
\end{proof}

We want to find a lower bound for $\cum_5(t)$.
In the following, we consider the behavior of the differences $\cum_j(t)-\cum_j(t+1)$ when a block of $\tL$s is appended to $t$.
We do so step by step, starting with $\cum_2(t)$.
Assume that $k\geq 1$ is an integer and set $t^{(k)}=2^kt+2^k-1$.
Note that by~\eqref{eqn_dt_rec} we have $\cum_{j}(t^{(k)}+1) = \cum_j(t+1)$.
By the recurrence~\eqref{eqn_coeff_2} we obtain
\begin{align*}
\cum_2\bigl(t^{(k)}\bigr)-\cum_2\bigl(t^{(k)}+1\bigr)
&=\frac{\cum_2\bigl(t^{(k-1)}\bigr)+\cum_2(t+1)}2+1-\cum_2(t+1)
\\&=\frac{\cum_2\bigl(t^{(k-1)}\bigr)-\cum_2(t+1)}2+1,
\end{align*}
which gives by induction
\begin{equation}\label{eqn_cum2_1k}
\begin{aligned}
\cum_2\bigl(t^{(k)}\bigr)-\cum_2\bigl(t^{(k)}+1\bigr)
&=\frac{\cum_2(t)-\cum_2(t+1)}{2^k}+\frac{2^k-1}{2^{k-1}}
\\&=2+\mathcal O\bigl(2^{-k}\bigr).
\end{aligned}
\end{equation}
We proceed to $\cum_3(t)$. For $k\geq 1$, we have
\begin{align*}
\cum_3\bigl(t^{(k)}\bigr)-\cum_3\bigl(t^{(k)}+1\bigr)
&=\frac{\cum_3\bigl(t^{(k-1)}\bigr)-\cum_3(t+1)}2
+3+\mathcal O\bigl(2^{-k}\bigr)
\end{align*}
by~\eqref{eqn_coeff_3} and~\eqref{eqn_cum2_1k}.
By induction and the geometric series we obtain
\begin{equation}\label{eqn_cum3_1k}
\begin{aligned}
\cum_3\bigl(t^{(k)}\bigr)-\cum_3\bigl(t^{(k)}+1\bigr)
&=\frac{\cum_3(t)-\cum_3(t+1)}{2^k}
+6+\mathcal O\bigl(k2^{-k}\bigr)
\\&=6+\mathcal O\bigl(k2^{-k}\bigr).
\end{aligned}
\end{equation}

Concerning $\cum_4(t)$, we have by~\eqref{eqn_coeff_4},~\eqref{eqn_cum2_1k}, and~\eqref{eqn_cum3_1k}
\begin{align*}
\hspace{3em}&\hspace{-3em}
\cum_4\bigl(t^{(k)}\bigr)-\cum_4\bigl(t^{(k)}+1\bigr)
=\frac{\cum_4\bigl(t^{(k-1)}\bigr)-\cum_4(t+1)}2
+2\bigl(\cum_3\bigl(t^{(k-1)}\bigr)-\cum_3(t^{(k-1)}+1)\bigr)
\\&+\frac34\left(\cum_2\bigl(t^{(k-1)}\bigr)-\cum_2(t^{(k-1)}+1)\right)^2-2
\\&=
\frac{\cum_4\bigl(t^{(k-1)}\bigr)+\cum_4(t+1)}2
+12+\mathcal O(k2^{-k})
+\frac34\left(2+\mathcal O(2^{-k})\right)^2-2
\\&=
\frac{\cum_4\bigl(t^{(k-1)}\bigr)-\cum_4(t+1)}2
+13+\mathcal O(k2^{-k})
\end{align*}
and by induction we obtain

\begin{equation}\label{eqn_cum4_1k}
\cum_4\bigl(t^{(k)}\bigr)-\cum_4\bigl(t^{(k)}+1\bigr)
=26+\mathcal O\bigl(k^22^{-k}\bigr).
\end{equation}

Finally, we have by~\eqref{eqn_coeff_5},~\eqref{eqn_cum2_1k},~\eqref{eqn_cum3_1k}, and~\eqref{eqn_cum4_1k}
\[
\begin{aligned}
\hspace{1em}&\hspace{-1em}
\cum_5\bigl(t^{(k)}\bigr)-\cum_5\bigl(t^{(k)}+1\bigr)
=\frac{\cum_5\bigl(t^{(k-1)}\bigr)-\cum_5(t+1)}2
+\frac 52\Bigl(\cum_4\bigl(t^{(k-1)}\bigr)-\cum_4(t^{(k-1)}+1)\Bigr)
\\
&+\frac52\Bigl(\cum_2\bigl(t^{(k-1)}\bigr)-\cum_2(t^{(k-1)}+1)\Bigr)\Bigl(\cum_3\bigl(t^{(k-1)}\bigr)-\cum_3(t^{(k-1)}+1)\Bigr)
\\&
-10\Bigl(\cum_2\bigl(t^{(k-1)}\bigr)-\cum_2(t^{(k-1)}+1)\Bigr)
=
\frac{\cum_5\bigl(t^{(k-1)}\bigr)-\cum_5(t+1)}2
+65+\mathcal O\bigl(k^22^{-k}\bigr)\\&
+\frac52\bigl(2+\mathcal O\bigl(2^{-k}\bigr)\bigr)\bigl(6+\mathcal O\bigl(k2^{-k}\bigr)\bigr)
-20+\mathcal O\bigl(2^{-k}\bigr)
\\&=
\frac{\cum_5\bigl(t^{(k-1)}\bigr)-\cum_5(t+1)}2
+75+\LandauO\bigl(k2^{-k}\bigr).
\end{aligned}
\]
and therefore by induction
\begin{equation}\label{eqn_cum5_1k}
\cum_5\bigl(t^{(k)}\bigr)-\cum_5\bigl(t^{(k)}+1\bigr)
=150+\mathcal O\bigl(k^32^{-k}\bigr).
\end{equation}

\begin{proposition}\label{prp_cum5_lower_bound}
Let $k\geq 1$ be an integer.
Assume that the integer $t\geq 1$ has $N_0$ inner blocks of zeros of length $\geq 2$, and $N_1$ blocks of $\tL$s of length $\leq k$. Define $N=N_0+N_1$.
If $N_2$ is the number of blocks of $\tL$s of length $>k$, we have
\[
\cum_5(t)\geq 150N_2-C\bigl(N+N_2k^32^{-k}\bigr)
\]
with an absolute constant $C$.
\end{proposition}
\begin{proof}
We proceed by induction on the number of blocks of $\tL$s in $t$.
The statement obviously holds for $t=0$.
Clearly, by the identity $\cum_5(2t)=\cum_5(t)$ we may append $\tO$s, preserving the truth of the statement (note that $N$ and $N_2$ are unchanged, since we only count \emph{inner} blocks of $\tO$s).
We therefore consider, for $r\geq 1$, appending a block of the form $\tO\tL^r$ to $t$, obtaining $t'=2^{r+1}t+2^r-1$.
Define the integers $N'$ and $N_2'$ according to this new value $t'$.
If $t$ is even, an additional block of zeros of length $\geq 2$ appears,
therefore $N'\geq N+1$, moreover $N_2'\leq N_2+1$.
By the bound $\lvert \cum_5(m+1)-\cum_5(m)\rvert\leq 240$ from~\eqref{eqn_a5_difference_bound}, $\cum_5(2n)=\cum_5(n)$, and the induction hypothesis we have
\begin{equation}\label{eqn_cum5_caseOOL}
\begin{aligned}
\cum_5\bigl(t'\bigr)&
=\bigl(\cum_5(t')-\cum_5(2t+1)\bigr)
+\bigl(\cum_5(2t+1)-\cum_5(t)\bigr)
+\cum_5(t)
\\&\geq \cum_5(t)-480
\geq 150N_2-C\bigl(N+N_2k^32^{-k}\bigr)+480
\\&\geq 150N_2-C\bigl(N'+N_2'k^32^{-k}\bigr)
\end{aligned}
\end{equation}
if $C$ is chosen large enough.
The case of odd $t$ remains.
The integer $t$ ends with a block of $\tL$s of length $s\geq 1$. We distinguish between three cases.
First, let $r\leq k$.
In this case, $N'=N+1$ and $N_2'=N_2$,
and reusing the calculation~\eqref{eqn_cum5_caseOOL} yields the claim.

In the case $r>k$, we have $N'=N$ and $N_2'=N_2+1$.
This case splits into two subcases.
Assume first that $s\neq k$.
We first consider the integer $t''=2t+1$.
The quantities $N''$ and $N_2''$ corresponding to the integer $t''$ satisfy $N''=N$ and $N_2''=N_2$ due to the restriction $s\neq k$,
and by hypothesis --- recall that the induction is on the number of blocks of $\tL$s in $t$ --- we have
\begin{equation}\label{eqn_tpp_hypothesis}
\cum_5(2t+1)\geq 150N_2-C\bigl(N+N_2k^32^{-k}\bigr).
\end{equation}
In this case, we need to extract the necessary gain of $150$ from~\eqref{eqn_cum5_1k}:
this formula yields together with~\eqref{eqn_tpp_hypothesis}
\begin{align*}
\cum_5(t')&=\cum_5(2t+1)+150+\mathcal O\bigl(k^32^{-k}\bigr)
\\&\geq 150N'_2-C\bigl(N'+N_2'k^32^{-k}\bigr)
\end{align*}
if $C$ is chosen appropriately.
Finally, we consider the subcase $s=k$, and again we set $t''=2t+1$ and choose $N''$ and $N_2''$ accordingly.
Here we have $N''=N-1=N'-1$ and $N_2''=N_2+1=N_2'$, and therefore by hypothesis
\[\cum_5(2t+1)\geq 150N_2-C\bigl((N-1)+(N_2+1)k^32^{-k}\bigr).\]
By the bound~\eqref{eqn_a5_difference_bound} we have
\begin{align*}
\cum_5(t')&\geq \cum_5(2t+1)-240\geq 150N'_2-C\bigl(N'+N_2'k^32^{-k}\bigr).
\end{align*}

This finishes the proof of Proposition~\ref{prp_cum5_lower_bound}.
\end{proof}
\subsection{Finishing the proof of the main theorem}
\label{sec:endmainproof}

By Lemma~\ref{lem_ct_uncritical} there is a constant $D_0$ such that $c_t>1/2$ if $D(t)\geq D_0$.
Assume that $C$ is the constant from Proposition~\ref{prp_cum5_lower_bound} and choose $k$ large enough such that $Ck^32^{-k}\leq 20$.
Choose $B=B(D_0,k)$ as in Corollary~\ref{cor_exceptions} and assume that $D(t)\leq D_0$.
The number $N_0$ of inner blocks of $\tO$s of length $\geq 2$ in $t$ and the number $N_1$ of blocks of $\tL$s of length $\leq k$ in $t$ are bounded by $B$ by this corollary.
Furthermore, recall that $M=N_1+N_2$, where $N_2$ is the number of blocks of $\tL$s of length $> k$.
Therefore by Proposition~\ref{prp_cum5_lower_bound},
\[\cum_5(t)\geq 130M-CB.\]
If $t$ contains sufficiently many blocks of $\tL$s, we therefore have by Lemmas~\ref{lem_a2_lower_bound} and~\ref{lem_cum4_upper_bound}
\begin{align*}
53\cum_2(t)
-4\cum_4(t)
+\frac25\cum_5(t)
&\geq 53M-104(M+1)+52M-\frac45CB
\\&= M-\frac45CB-104.
\end{align*}
For large $M$ this is positive, and by~\eqref{eqn_lincomb_sufficient} it follows that $c_t>1/2$ for sufficiently many (greater than some absolute bound) blocks of $\tL$s. The proof is complete.
\section{Normal distribution of \texorpdfstring{$\delta(j,t)$}{delta(j,t)}}
In this section we prove Theorem~\ref{thm_normal}.
By~\eqref{eqn_gamma_integral} we have
\[
\delta(j,t)=\int_{-1/2}^{1/2} \gamma_t(\vartheta)\e(-j\vartheta)\,\mathrm d\vartheta.
\]
As above, we truncate the integral at $\pm\vartheta_0$,
where
\[\vartheta_0=M^{-1/2}R,\]
$M=2M'+1$ is the number of blocks of $\tL$s in $t$,
and $R$ is chosen later.
In analogy to the reasoning above, we assume that

\begin{equation*}
\begin{aligned}
8&\leq R\leq M^{1/6}\quad\mbox{and}\quad
\vartheta_0\leq \frac 1\tau.
\end{aligned}
\end{equation*}
Again, by our choice of $R=\log M$ below, this will be satisfied for a sufficiently large number~$M$ of blocks.
We define a coarser approximation of $\gamma_t(\vartheta)$ than used for the proof of our main theorem, as it is sufficient to derive the normal distribution-statement.
Let
\begin{equation*}
\begin{aligned}
\gamma^{(2)}_t(\vartheta)&=
\exp\left(-\cum_2(t)\frac{(\tau\vartheta)^2}2\right),\\
\widetilde\gamma^{(2)}_t(\vartheta)
&=\gamma_t(\vartheta)-\gamma^{(2)}_t(\vartheta).
\end{aligned}
\end{equation*}
The proof of the following estimate essentially only requires to change some numbers in the proof of Proposition~\ref{prp_tildegamma_est} and we leave it to the interested reader.
\begin{proposition}\label{prp_tildegamma_est2}
There exists an absolute constant $C$
such that we have
\begin{equation*}
\bigl\lvert\widetilde\gamma^{(2)}_t(\vartheta)\bigr\rvert
\leq CM\vartheta^3
\end{equation*}
for $\lvert\vartheta\rvert\leq \min\bigl(M^{-1/3},\tau^{-1}\bigr)$,
where $M$ is the number of blocks of $\tL$s in $t$.
\end{proposition}
Noting that $\vartheta_0\leq M^{-1/3}$ and $\vartheta_0\leq 1/\tau$ for large $M$, we obtain
from Lemma~\ref{lem_gamma_tail} and Proposition~\ref{prp_tildegamma_est2}
\begin{equation*}
\begin{aligned}
\delta(j,t)&=\int_{-\vartheta_0}^{\vartheta_0} \gamma_t(\vartheta)\e(-jt)\,\mathrm d\vartheta
+\LandauO\left(\exp(-R^2/4)\right)
\\&=
\int_{-\vartheta_0}^{\vartheta_0} \gamma^{(2)}_t(\vartheta)\e(-j\vartheta)\,\mathrm d\vartheta
+\LandauO\left(M^{-1}R^4\right)
+\LandauO\left(\exp(-R^2/4)\right)
\end{aligned}
\end{equation*}
if only $M$ is large enough and $R\leq M^{1/6}$.
We extend the integral to $\mathbb R$, introducing an error
\[
\int_{\tau\vartheta_0}^\infty
\exp\bigl(-\cum_2(t)\vartheta^2/2\bigr)
\,\mathrm d\vartheta
\ll
\exp\bigl(-\cum_2(t)R^2/(2M)\bigr)
\leq
\exp\bigl(-R^2/2\bigr)
\]
since $\cum_2(t)\geq M$ by Lemma~\ref{lem_a2_lower_bound}.
We obtain the representation
\begin{equation*}
\begin{aligned}
\delta(j,t)&=\int_{-\infty}^{\infty} \exp\bigl(-\cum_2(t)(\tau\vartheta)^2/2\bigr)\e(-j\vartheta)\,\mathrm d\vartheta
+\LandauO(E)\\
&=
\frac 1\tau\int_{-\infty}^{\infty} \exp\bigl(-\cum_2(t)\vartheta^2/2-ij\vartheta\bigr)\,\mathrm d\vartheta
+\LandauO(E)
\end{aligned}
\end{equation*}
for large enough $M$ and $R\leq M^{1/6}$,
where
\[E=M^{-1}R^4+\exp\bigl(-R^2/4\bigr).\]

Now, we choose $R=\log M$. Our hypothesis $R\geq 8$ implies $\exp\bigl(-R^2/4\bigr)\leq M^{-1}$
and therefore
\[E\ll M^{-1}\bigl(\log M\bigr)^4.\]

The appearing integral can be evaluated by completing to a square and evaluating a complete Gauss integral: 
\[
-\cum_2(t)\vartheta^2/2-ij\vartheta
=
-\left((\cum_2(t)/2)^{1/2}\vartheta+\frac{ij}{\sqrt{2}\cum_2(t)^{1/2}}\right)^2
-\frac{j^2}{2\cum_2(t)}.
\]
The imaginary shift is irrelevant due to the residue theorem, and after inserting the Gauss integral and slight rewriting we obtain the theorem.

\subsection*{Data availability statement}
The datasets generated and analysed during the current study are available from the corresponding author on reasonable request.

\bibliographystyle{amsplain}
\bibliography{SpiegelhoferWallner}

\bigskip
\begin{center}
\begin{tabular}{c}
Department Mathematics and Information Technology,\\
Montanuniversit\"at Leoben,\\
Franz-Josef-Strasse 18, 8700 Leoben, Austria\\
lukas.spiegelhofer@unileoben.ac.at\\
ORCID iD: 0000-0003-3552-603X
\end{tabular}
\end{center}

\smallskip
\begin{center}
\begin{tabular}{c}
Institute of Discrete Mathematics and Geometry,\\
TU Wien,\\
Wiedner Hauptstrasse 8--10, 1040 Wien, Austria\\
michael.wallner@tuwien.ac.at\\
ORCID iD: 0000-0001-8581-449X
\end{tabular}
\end{center}
\end{document}